\documentclass[12pt]{article}
\usepackage[a4paper,margin=2.5cm,includefoot]{geometry}

\usepackage[backend=bibtex,
            style=numeric,
            citestyle=numeric,
            sorting=nyt,
            maxbibnames=99,
            doi=false,
            url=false,
            eprint=false,
           ]{biblatex}

\usepackage{amsmath,amssymb,amsfonts,amsthm}
\usepackage{graphicx}
\usepackage{xcolor}
\usepackage{hyperref}
\usepackage{enumitem}
\usepackage{mathrsfs}
\usepackage{accents}
\usepackage{mathtools}
\usepackage{authblk}
\usepackage{booktabs,multirow}
\usepackage{float}
\usepackage{pdflscape}   % rotates a page in the PDF viewer
\usepackage{caption}     % for captionof if needed

\title{The Law of Large Numbers and CLT for Non-stationary Markov Jump Processes Exhibiting Time-of-Day Effects}
\author{Monte Fischer\thanks{m0nte@stanford.edu} }
\author{Peter W.~Glynn}
\affil{Department of Management Science \& Engineering, Stanford University, Stanford, CA 94305, USA}
\date{October 13, 2025}

\newcommand{\ceil}[1]{\left\lceil #1 \right\rceil}

\theoremstyle{plain}
\newtheorem{theorem}{Theorem}

\newtheorem{proposition}{Proposition}
\theoremstyle{definition}

\numberwithin{condition}{section}
\theoremstyle{remark}
\newtheorem{remark}{Remark}

\newtheoremstyle{assumptionstyle}
  {\topsep}   
  {\topsep}   
  {\normalfont}
  {0pt}       
  {\bfseries} 
  {}          
  {  }         
  {\thmname{#1}~\thmnumber{#2}~} 
\theoremstyle{assumptionstyle}
\newtheorem{assumption}{Assumption}

\newcommand{\p}{\mathbb{P}}
\newcommand{\E}{\mathbb{E}}
\newcommand{\Var}{\mathrm{Var}}

\newcommand{\R}{\mathbb{R}}
\newcommand{\Z}{\mathbb{Z}}

\newcommand{\GG}{\mathscr{G}}
\newcommand{\HH}{\mathscr{H}}
\newcommand{\KK}{\mathscr{K}}

\newcommand{\bigO}{O}

\newcommand{\convp}{\overset{P}{\to}}
\newcommand{\as}{\text{a.s.}}
\newcommand{\convas}{\overset{\as}{\longrightarrow}}

\newcommand{\eqd}{\overset{d}{=}}

\newcommand{\dd}[1]{\frac{d}{d {#1}}}
\newcommand{\pp}[1]{\frac{\partial}{\partial {#1}}}

\newcommand{\ubar}[1]{\underaccent{\bar}{#1}}

\newcommand{\exitrate}{\lambda}
\newcommand{\maxexitrate}{{\bar \lambda}}
\newcommand{\rankonestoch}{\Psi}

\usepackage[normalem]{ulem}

\begin{document}

\maketitle

\vspace{-1em}

\begin{abstract}
In this paper, we develop a general law of large numbers and central limit theorem for cumulative reward processes associated with finite state Markov jump processes with non-stationary transition rates. Such models commonly arise in service operations and manufacturing applications in which time-of-day, day-of-week, and secular effects are of first-order importance in predicting system behavior. Our theorems allow for non-stationary reward environments that continuously accumulate reward, while also including contributions from non-stationary lump-sum rewards of random size that are collected at either jump times of the underlying process, jump times of a Poisson process modulated by the underlying process, or scheduled deterministic times. As part of our development, we also obtain a new central limit theorem for the special case in which the jump process transition rates and reward structure are periodic (as may occur over a weekly time interval), as well as for jump process models with resetting. We include a simulation study illustrating the quality of our CLT approximations for several non-stationary stochastic models.
\end{abstract}

\section{Introduction}
In many service operations settings, time-of-day effects play a first-order role in determining customer quality-of-service and other managerially significant performance objectives. Such settings lead naturally to Markov jump process (or, equivalently, continuous-time Markov chain) models $X = (X(t): t \ge 0)$ in which the transition rates are non-stationary and depend on time-of-day, day-of-week, or seasonal effects. Furthermore, there may be secular effects present that generate long-term changes in customer demand for a product (perhaps induced by a targeted marketing campaign) or changes in the product mix handled by a service center. These effects are represented in the model via a non-stationary specification of the transition rates $Q(t,x,y)$ describing the rate at which $X$ jumps from $x$ to $y$ at time $t$ (for $x \ne y$). This leads to a non-stationary specification of the rate matrices $(Q(t): t \ge 0)$ that describe the transition dynamics of $X$ as a function of time $t$. The applied importance of this non-stationary environment is reflected in the survey of \cite{whitt2018time}, focused on theory relevant to the analysis of queueing models.

In this paper, we develop the law of large numbers (LLN) and central limit theorem (CLT) for general cumulative reward processes associated with $X$. These results provide managerially useful approximations to the distribution of the total reward for a system accumulated over a given long decision horizon $[0,t]$, so that, for example, one can compute the probability that total profit will exceed some break-even value. We permit the rate $r(t,x)$ at which reward accumulates in state $x$ to depend on $t$ itself. We also allow for ``jump time'' lump-sum rewards that are received when $X$ jumps from state $x$ to state $y$ at time $T$, yielding an instantaneous random reward of $\mathscr G(T,x,y)$. In addition, we permit ``scheduled rewards'' that are collected at deterministic times (e.g.~dividends or quarterly payments) through the use of random variables (rv's) $\HH(t_i, x)$ to denote the lump-sum reward received at time $t_i$ when $X$ occupies state $x \in S$. Finally, we let $N = (N(t): t \ge 0)$ be a non-stationary Poisson process driven by $X$ with intensity $(\beta(t, X(t)): t \ge 0)$ and let $\KK(\Lambda, x)$ be an ``external reward'' collected at the Poisson event time $\Lambda$ when $X$ occupies state $x \in S$. The rv's $\GG(T,x,y)$, $\HH(t_i,x)$, and $\KK(\Lambda,x)$ have corresponding cumulative distribution functions (cdf's) $G(T,x,y,\cdot)$, $H(t_i,x,\cdot)$, and $K(\Lambda,x,\cdot)$. Because we allow $r(t,x)$, $\GG(t,x,y)$, $\HH(t,x)$, and $\KK(t,x)$ to depend on $t$, our LLN and CLT will provide new approximations even for Markov jump processes with stationary transition rates, for which there is a single rate matrix $Q$ such that $Q(t) \equiv Q$ for $t \ge 0$. By permitting the cdf $G(T, x, y, \cdot)$ to possess a jump discontinuity at 0, we can model situations in which jump-time rewards are generated stochastically at a subsequence of the jump times of $X$.

Let $T_j$ be the time of the $j$th jump of $X$ and set $Y_j = X(T_j)$ for $j \geq 1$, with $Y_0 = X(0)$; we take $X$ to be right continuous throughout this paper. We develop approximations for the cumulative reward
\begin{equation}
\label{eqn:reward_functional}
    R(t) \triangleq \int_0^t r(s, X(s)) ds + \sum_{j=1}^{J(t)} \GG(T_j, Y_{j-1}, Y_j) + \sum_{i=1}^{n(t)} \HH(t_i, X(t_i)) + \sum_{k=1}^{N(t)} \KK(\Lambda_k, X(\Lambda_k)),
\end{equation}
where $J(t)$ is the total number of jumps of $X$ over $[0,t]$, $N$ has event times $(\Lambda_k : k \ge 1)$, and $0 = t_0 < t_1 < t_2 < \dots < t_{i-1} < t_i < \dots$ is a sequence of deterministic times for which $t_i \to \infty$, with $n(t) \triangleq \sup\{ i \ge 0: t_i \le t\}$. 

Our law of large numbers (Theorem \ref{thm:slln}) shows that
\begin{equation}
\label{eqn:intro_lln}
    \frac{R(t)}{\E R(t)} \convas 1
\end{equation}
as $t \to \infty$. Our central limit theorem (Theorem \ref{thm:clt}) refines $\eqref{eqn:intro_lln}$ to 
\begin{equation}
\label{eqn:intro_clt}
    \frac{R(t) - \E R(t)}{\sqrt{\Var R(t)}} \Rightarrow \mathcal{N}(0,1)
\end{equation}
as $t \to \infty$, where $\mathcal{N}(0,1)$ is a standard normal rv with mean 0 and unit variance, and $\Rightarrow$ denotes weak convergence. In view of \eqref{eqn:intro_clt}, we can then approximate the distribution of $R(t)$ as
\begin{equation*}
    R(t) \overset{D}{\approx} \E R(t)+ \sqrt{\Var R(t)} \mathcal{N}(0,1)
\end{equation*}
when $t$ is large, and $\overset{D}{\approx}$ denotes ``has approximately the same distribution as'' and has no rigorous meaning other than that carried by \eqref{eqn:intro_clt}. This result is intended to support the view that even in the non-stationary setting, one can reasonably expect, in great generality, that cumulative rewards are approximately normal when the time horizon is large.

A significant mathematical complication in the non-stationary setting is that one cannot typically expect that $\E R(t) \sim mt$ or $\Var R(t) \sim \sigma^2 t$ as $t \to \infty$ for some values $m$ and $\sigma^2$, where $a(t) \sim b(t)$ as $t \to \infty$ means that $a(t) / b(t) \to 1$ as $t \to \infty$. For example, $Q(t)$ could alternate between $Q_1$ on intervals of the form $[2^{2i}, 2^{2i+1}]$ and $Q_2$ on $[2^{2i+1}, 2^{2i+2}]$ for $i \geq 0$. However, our assumptions will ensure that $\sigma^2(t) = \Theta(t)$ as $t \to \infty$, where $a(t) = \Theta(b(t))$ as $t \to \infty$ means that there exist constants $0 < c_1, c_2, t_0 < \infty$ such that $c_1 |b(t)| \le |a(t)| \le c_2 |b(t)| $ for all $t \ge t_0$. We allow this flexibility in order to communicate the ``model-free'' nature of the LLN and CLT we derive (i.e., we make no assumptions on how $Q(t)$ varies as a function of $t$). We formulate our theory for finite state (rather than infinite state) Markov jump processes. This choice covers many application settings and avoids the technical complications that arise in the infinite state non-stationary setting owing to, for instance, the possibility of alternating periods of recurrence and transience. Our framework accommodates a highly general, non-stationary reward structure, including lump-sum rewards of random size at both jump times and deterministic time points.

The main contributions of this paper are:
\begin{enumerate}[label=(\arabic*)]
\item a general law of large numbers (Theorem \ref{thm:slln}) and central limit theorem (Theorem \ref{thm:clt}) for non-stationary finite state Markov jump processes with non-stationary continuously accruing rewards and lump-sum rewards that accumulate at jump times, external times, and scheduled times, with only measurability assumptions on how $Q(t)$ depends on $t$;

\item a new martingale representation (Section \ref{sec:mtg}) for the cumulative reward $R(t)$;

\item a new CLT (Theorem \ref{thm:clt_periodic}) for periodic Markov jump processes and a discussion of how to compute the centering and scaling constants for the periodic CLT.

\end{enumerate}
Our new contributions also include a discussion of the CLT for non-stationary Markov jump processes with resetting; see Section \ref{sec:reset}. Such an extension is intended to cover non-stationary applied settings in which the work is cleared regularly (e.g. an airport security checkpoint with limited hours where no work is pushed forward from one day into the next). Some recent work has investigated resetting mechanisms as part of server-side control or performance optimization strategies, as in \cite{bonomo2022mitigating, roy2024queues}, but these models treat resetting as a strategic stochastic intervention rather than as a structural feature of the system.

Dobrushin \cite{dobrushin1956centrali,dobrushin1956centralii} developed an early and important contribution to the CLT for non-stationary Markov chains in discrete time, also within the context of ``model-free'' assumptions on the one-step transition probabilities. Notably, these papers introduced the well-known Dobrushin coefficient as a tool for studying the ergodic properties of Markov chains. By contrast, we develop our theory in continuous time and we work with a much richer reward structure than do \cite{dobrushin1956centrali,dobrushin1956centralii}. These papers assume that the Markov chain is real-valued and that the rewards are given by the state values, so that unlike in our setting there is no time dependence permitted in the reward structure itself. In addition, lump-sum rewards do not appear in this discrete time theory. The recent paper \cite{vassiliou2020laws} developed a LLN (but not a CLT) in the discrete-time non-stationary finite state setting. However, \cite{vassiliou2020laws} assumed that the one-step transition matrices converge asymptotically, so that the Markov chain is ``asymptotically stationary''. This is a very strong requirement, especially relative to our model-free assumptions on our rate matrices $Q(t)$, and is often referred to as \textit{strong ergodicity}; see \cite{iosifescu1972two} and \cite{johnson1988conditions} for further discussion. The works \cite{breuer2002markov,vassiliou2020laws} provide LLNs for periodic Markov jump processes and Markov chains, respectively. Previous work by \cite{vassiliou1986asymptotic} analyzed asymptotic (in time) cross-sectional variability in a periodic Markov population process, but did not consider auto-correlation effects across time, as we do. Thus the present work, to our knowledge, provides the first CLT (with associated variance constant) for finite state periodic Markov models, whether in discrete time or continuous time. (We develop our Theorem 5 in continuous time, but its discrete time variant can be analogously derived.)

    To develop our martingale representation, we need to show that the non-stationary Markov jump process we consider ``loses memory'' in the sense that $X(t+s)$ becomes asymptotically independent of $X(t)$ when $s$ is large. Our result (Proposition \ref{prop:mixing}) is a non-stationary analog to the extensive theory on mixing that appears in the literature on stationary stochastic processes; see \cite{bradley2005mixing} for an accessible survey in the discrete-time stationary setting. Similar upper bounds on the rate at which asymptotic loss of memory (also known as \textit{weak ergodicity}) occurs in the setting of finite state non-stationary Markov jump processes have been studied in \cite{zeifman2015two}. These results are non-constructive and require verifying a condition for every rate matrix $Q(t)$, $t \ge 0$. In contrast, Proposition \ref{prop:mixing} supplies an explicit decomposition of the transition matrices with terms that can be evaluated via recursion and may be stated in terms of the uniform lower bounds $\ubar q$ on the rates in our assumption \ref{a3}. 
    
     In order to provide maximum modeling flexibility in our theory (so as to include cases where $Q(t)$ is not continuously differentiable or even continuous in t, as may be required at shift changes in a service center), we develop our theory at a maximum level of generality with regard to the behavior of $Q(t)$ as a function of $t$. We require only measurability (see our assumption \ref{a1}). Foundational analytical questions in the non-stationary continuous-time setting have only recently seen full resolution: the classical work of \cite{feller1940integro} provided conditions under which the Kolmogorov forward and backward equations hold for Markov jump processes with denumerable state spaces. More recently, \cite{ye2008existence,feinberg2014solutions} significantly weakened these conditions by showing existence results that require only measurability of the rate matrices.

    By contrast with the non-stationary theory developed in this paper, there is a large literature on the LLN and CLT for Markov chains and processes when the transition probabilities and rewards are stationary in time.  An early contribution to this literature was that of \cite{gordin1969central}, who showed how martingale ideas can be used to develop CLT's for stationary processes. This was followed by \cite{maigret1978theoreme}, who applied these ideas to Harris recurrent Markov chains in discrete time and made clear the central role of Poisson's equation in this setting. A recent treatment of limit theorems for stationary continuous-time Markov processes, with a focus on the uncountable state space case, was given in \cite{komorowski2012fluctuations}; see also \cite{kulik2017ergodic}. Lyapunov function assumptions guaranteeing the CLT can be found in \cite{glynn2024solution}, under close to minimal conditions. As far as we are aware, there is no prior literature on Markov models with stationary transition probabilities, but non-stationary rewards.

    In order to apply our LLN and CLT results in practical settings, it is necessary to compute the moments $\E R(t)$ and $\Var R(t)$. Vector-valued ordinary differential equations (ODEs) for the computation of $\E [R(t) \mid X(0)=x]$, $\Var [R(t) \mid X(0)=x]$ for $x \in S$ in the case of piecewise continuous non-stationary rates and rewards were developed in \cite{norberg1995differential}. Matrix-valued ODEs to compute moments of the form $\E [R^k(t) I(X(t)=y) \mid X(0)=x]$ for $x,y \in S$ were given in \cite{bladt2020matrix}, covering piecewise continuous non-stationary rates and rewards. Previous literature addressing the case of stationary rates can be found in \cite{glynn1984some,grassmann1987means,bladt2002distributions}. By contrast, our results provide integral equation representations for $\E R(t)$ and $\E R^2(t)$ in the case of measurable non-stationary rates and random lump-sum rewards.

    The computation of the exact distribution of $R(t)$ (as opposed to the asymptotic distribution as expressed through the CLT) for Markov jump processes with stationary transition rates was discussed in \cite{reibman1989markov,sericola1990closed,souza1998algorithm,tijms2000fast}. Such methods involve much more computation than does our CLT. Expressions for the mean and variance of $R(t)$ in the stationary setting were provided in \cite{bladt2002distributions}, covering both time-integrated and jump-time lump-sum rewards, along with two additional numerical methods for computing the exact distribution of $R(t)$ when the reward functional consists only of  time-integrated rewards. In \cite{bladt2020matrix}, these results were extended to cover non-stationary, piecewise continuous rates and rewards. An integro-PDE representation for the reward $R(t)$ in the non-stationary setting with deterministic lump-sum rewards, allowing for scheduled lump-sum rewards, was developed in \cite{hesselager1996probability}; we extend to an integro-PDE for stochastic lump-sum rewards in Section \ref{sec:computational_considerations}.

    This paper is organized as follows. In Section \ref{sec:mixing}, we state the basic assumptions of our theory, and study the mixing properties of the non-stationary Markov jump process $X$. Section \ref{sec:mtg} applies this theory to develop a martingale representation for $R(t)$ based on Poisson's equation that generalizes a representation that is known in the stationary setting. In Sections \ref{sec:lln} and \ref{sec:clt}, we state and prove our main results, namely the LLN and CLT for $R(t)$. We discuss the computation of $\E R(t)$ and $\Var R(t)$ for our general non-stationary reward structure in Section \ref{sec:computation_mean_var}, and discuss computational considerations involved in computing either the approximate (as related to the CLT) or exact distribution of $R(t)$ in Section \ref{sec:computational_considerations}. Section \ref{sec:periodic} focuses on the important special case of non-stationary jump processes in which the rates and rewards are periodic, and the associated simplifications of the asymptotic theory. We discuss non-stationary Markov jump processes with resetting in Section \ref{sec:reset}. The paper concludes with a simulation study that focuses on consideration of the quality of our CLT approximations for several non-stationary stochastic models in Section \ref{sec:numerics}.

\section{Mixing for Markov Jump Processes}
\label{sec:mixing}

We start by stating the assumptions that underlie our theory. Let $X \triangleq (X(t): t \geq 0)$ be an $S$-valued Markov jump process for which $|S| = d  < \infty$. We denote by $Q(t) \triangleq (Q(t, x, y): x,y \in S)$ the rate matrix describing the dynamics of $X$ at time $t$. 
Given $X$, let $N \triangleq (N(t) : t \ge 0)$ be a Poisson counting process with intensity $(\beta(t,X(t)) : t \ge 0)$ and event times $(\Lambda_k : k \ge 1)$. Conditional on $X$ and $N$, we assume that $\big((\GG(T_k, Y_{k-1}, Y_k), \HH(t_k, X(t_k)), \KK(\Lambda_k, X(\Lambda_k))\big) : k \ge 1)$ is a collection of independent rv's for which 
\begin{align*}
    \p(\GG(T_k, Y_{k-1}, Y_k) \le z \mid X, N) &= G(T_k, Y_{k-1}, Y_k, z), \\
    \p(\HH(t_k, X(t_k)) \le z \mid X,N) &= H(t_k, X(t_k), z),\\
    \p(\KK(\Lambda_k, X(\Lambda_k)) \le z \mid X,N) &= K(\Lambda_k, X(\Lambda_k), z),
\end{align*}
where $(G(t,x,y,\cdot) : t \ge 0, x,y \in S)$, $(H(t,x,\cdot) : t \ge 0, x \in S)$, and $(K(t,x,\cdot) : t \ge 0, x \in S)$ are given families of cdf's.
We make the following assumptions about the rate matrices $Q(t)$, the reward rate function $r(t) \triangleq (r(t,x): x \in S)$, the intensity function $\beta$ of the external lump-sum rewards, and the lump-sum reward distributions $G$, $H$, and $K$. Throughout this paper, we adopt the convention that real-valued functions defined on $S$ (such as $r(t, \cdot)$) are column vectors, while probability distributions defined on $S$ are row vectors. 
\begin{assumption}
\label{a1}
    For $x,y \in S$ and $z \in \R$, $Q(\cdot, x,y)$, $r(\cdot, x)$, $G(\cdot,x,y,z)$, $K(\cdot, x,z)$, and $\beta(\cdot, x)$ are measurable.
\end{assumption}
\begin{assumption}
\label{a2}
    $\mathcal{B} \triangleq \{(x,y):  Q(t,x,y) > 0\}$ is independent of $t$ and $Q(1)$ (and hence $Q(t)$) is an irreducible rate matrix. All future references to $\mathcal{B}$ are made under this assumption.
\end{assumption}
\begin{assumption}
\label{a3}
    For each $(x,y) \in \mathcal{B}$,
    \begin{equation*}
        0 < \ubar q(x,y) \triangleq \inf_{t \geq 0} Q(t,x,y) \leq \sup_{t \geq 0} Q(t,x,y) \triangleq \bar q(x,y) < \infty.
    \end{equation*}
\end{assumption}
\begin{assumption}
\label{a4}
There exists a finite-valued function $c = (c(x) : x \in S)$ such that
\begin{align*}
    0 \le \ubar r(x) \triangleq \inf_{t \geq 0} r(t,x) &\leq \sup_{t \geq 0} r(t,x) \triangleq \bar r(x) < \infty, \displaybreak[1] \\
    0 \le \ubar \beta(x) \triangleq \inf_{t \geq 0} \beta(t,x) &\leq \sup_{t \geq 0} \beta(t,x) \triangleq \bar \beta(x) < \infty, \displaybreak[1]\\
    0 = \sup_{t \ge 0} G(t,x,y,0{\scriptstyle -}) &\le \inf_{t \ge 0} G(t,x,y,c(x)) = 1, \displaybreak[1] \\
    0 = \sup_{t \ge 0} H(t,x,0{\scriptstyle -}) &\le \inf_{t \ge 0} H(t,x,c(x)) = 1, \displaybreak[1]\\
    0 = \sup_{t \ge 0} K(t,x,0{\scriptstyle -}) &\le \inf_{t \ge 0} K(t,x,c(x)) = 1,
\end{align*}
for all $x \in S$.
Furthermore, with 
\begin{align*}
    \gamma_1(t,x) &\triangleq \sum_{(x,y) \in \mathcal B} Q(t,x,y) \int_0^\infty z G(t,x,y,dz),  \\
    \gamma_2(t,x) &\triangleq \beta(t,x) \int_0^\infty z K(t,x,dz),  \\
    \gamma(t,x) &\triangleq \gamma_1(t,x) + \gamma_2(t,x),
\end{align*}
and 
\begin{equation}
    \ubar \gamma(x) \triangleq \inf_{t \ge 0} \gamma(t,x),
\end{equation}
it holds that 
\begin{equation}
    \label{eqn:a4_positive_reward_accrual}
    \ubar r(x) + \ubar \gamma(x) > 0
\end{equation}
for all $x \in S$.

\end{assumption}

\begin{assumption}
\label{a5}
The sequence $(t_i : i \geq 0)$ is such that
\begin{equation*}
    \ubar t \triangleq \inf \{ t_i - t_{i-1} : i \geq 1 \} > 0.
\end{equation*}
\end{assumption}

The key assumptions on the dynamics of $X$ are \ref{a2} and \ref{a3}. As in the case of stationary transition probabilities, reducible systems can typically be reduced to the analysis of their irreducible sub-systems. Assumption \ref{a3} rules out settings in which some states $x$ become close to absorbing (because the outgoing rates $Q(t,x,\cdot)$ are approaching zero, or the incoming rates $Q(t, \cdot, x)$ are becoming large), and others are becoming close to transient (because their incoming rates are approaching zero). Regarding \ref{a4}, note that \eqref{eqn:a4_positive_reward_accrual} can always be achieved by simply adding a nonzero deterministic multiple of $t$ to $R(t)$. The boundedness hypotheses on the rewards are generally reasonable from a modeling perspective, and rule out pathologies (for example, in which $r(t, \cdot)$ grows quickly enough that $R(t)$ is dominated by reward from the last few time periods, precluding a normal approximation). 

Under \ref{a1}-\ref{a3}, it is known that there exists a unique family of stochastic matrices $(P(s,t): 0 \leq s \leq t < \infty)$ for which
\begin{equation}
\label{eqn:kbe_matrix}
    P(t-s, t) = I + \int_0^s Q(t-u) P(t-u, t) du
\end{equation}
for $0 \leq s \leq t$, and
\begin{equation}
\label{eqn:kfe_matrix}
    P(t, t+s) = I + \int_0^s P(t,t+u) Q(t+u) du
\end{equation}
for $s \geq 0$; see the early work of \cite{feller1940integro} for the case of continuous rates and the extension to the case of measurable rates in \cite{ye2008existence,feinberg2014solutions}. The integral equations \eqref{eqn:kbe_matrix} and \eqref{eqn:kfe_matrix} are known respectively as the \textit{Kolmogorov backward} and \textit{forward equations} associated with $X$. Furthermore, for each $x,y \in S$, $P(s,t, x, y)$ is absolutely continuous in $s$ and $t$ for $0 \leq s \leq t$, so that there exist measurable functions $\pp s P(s,t,x,y)$ and $\pp t P(s,t,x,y)$ such that 
\begin{equation*}
    P(t-r, t) - I = \int_0^r \pp s P(t-u,t) du
\end{equation*}
for $0 \leq r \leq t$, and
\begin{equation*}
    P(s, s+r) - I = \int_0^r \pp t P(s,s+u) du
\end{equation*}
for $r \geq 0$.
In addition, if there exist times $0 = v_0 < v_1 < v_2 < \dots$ with $v_i \to \infty$ such that $Q(\cdot)$ is continuous on $(v_i,v_{i+1})$ with left limits at $v_i$ and right limits at $v_{i+1}$ for $i \geq 0$, then $\pp s P(s,t)$ exists as a partial derivative at all $s \notin \mathscr{D}[0,t] \triangleq \{v_i : i \geq 0, v_i \leq t\}$, and $\pp t P(s,t)$ exists as a partial derivative at all $t \notin \mathscr{D}[t,\infty) \triangleq \{v_i : i \geq 0, v_i \geq t\}$. It follows that 
\begin{equation}
\label{eqn:kbe_diff}
    \frac{\partial}{\partial s} P(s,t) = -Q(s) P(s,t)
\end{equation}
for $s \notin \mathscr{D}[0,t]$, and
\begin{equation}
\label{eqn:kfe_diff}
    \pp t P(s,t) = P(s,t) Q(t)
\end{equation}
for $s \notin \mathscr{D}[t,\infty)$.

Our first result describes the mixing behavior of $X$ that plays a fundamental role in establishing our LLN and CLT. 
\begin{proposition}
\label{prop:mixing}
Assume \ref{a1}-\ref{a3}. For any $u_0 > 0$, there exists $\delta > 0$, positive stochastic matrices $(\rankonestoch_k(s):s \ge 0)$ with identical rows, and stochastic matrices $(\mathscr{A}_k(s): s \ge 0)$ such that
\begin{equation}
\label{eqn:mixing_equation}
    P(s, s+k u_0) = (1-(1-\delta)^k)\,\rankonestoch_k(s) + (1-\delta)^k \mathscr{A}_k(s)
\end{equation}
for $s \ge 0$ and $k \in \Z_+$.
\end{proposition}
\begin{remark} We observe that Proposition \ref{prop:mixing} implies that $X$ exhibits weak ergodicity, in the sense that for each $x,y \in S, \|P(s,s+ku_0, x, \cdot) - P(s,s+ku_0, y, \cdot)\|_{\mathrm{tv}} \to 0$ as $k \to \infty$, where $\|v_1 - v_2\|_{\mathrm{tv}} \triangleq \sum_x |v_1(x)-v_2(x)|$ is the total variation distance between two probability mass functions $v_1, v_2$ on $S$. This asymptotic loss of memory will yield the LLN and CLT for $R(t)$.
\end{remark}

\begin{proof}[Proof of Proposition \ref{prop:mixing}]
Put $\exitrate(t,x) \triangleq -Q(t,x,x)$ and let $\maxexitrate \triangleq \sup\{\exitrate(t,x) : x \in S, t \geq 0\}$. Assumption \ref{a3} ensures that $\maxexitrate < \infty$. Hence, for $t \geq 0$, we may write $Q(t) = -\maxexitrate(I - B(t))$ where $B(t) = I + \maxexitrate^{-1} Q(t)$ is stochastic. We have by \ref{a2} that $B(t) \geq G$ for $G$ an irreducible non-negative matrix.

Equations \eqref{eqn:kbe_matrix} and \eqref{eqn:kfe_matrix} imply that for a.e.~$w \ge 0$,
\begin{equation*}
    \pp t P(s,s+w) = P(s,s+w) Q(s+w),
\end{equation*}
so that
\begin{equation*}
    e^{\maxexitrate r} \pp t P(s,s+w) + \maxexitrate e^{\maxexitrate w} P(s,s+w) = \maxexitrate e^{\maxexitrate w} P(s,s+w) R(s+w)
\end{equation*}
for a.e.~$w \ge 0$. Consequently,
\begin{equation}
\label{eqn:mixing_exp_derivative}
    \dd w (e^{\maxexitrate w} P(s,s+w)) = \maxexitrate e^{\maxexitrate w} P(s,s+w)R(s+w)
\end{equation}
for a.e.~$w$. Since the absolute continuity of $P(s, s+\cdot)$ implies that $e^{\maxexitrate w} P(s,s+w)$ is absolutely continuous in $w$ (see, for example, \cite[p.~111]{royden1988}), \eqref{eqn:mixing_exp_derivative} yields
\begin{equation*}
    e^{\maxexitrate w} P(s,s+w) - I = \int_0^w \maxexitrate e^{\maxexitrate u} P(s,s+u)B(s+u)du
\end{equation*}
for $w \geq 0$. Consequently,
\begin{equation}
\label{eqn:mixing_iterative_base}
    P(s,s+w) = e^{-\maxexitrate w} I+ \int_0^w \maxexitrate e^{\maxexitrate(u-w)} P(s,s+u)B(s+u) du
\end{equation}

If we now iterate \eqref{eqn:mixing_iterative_base} (i.e.~use \eqref{eqn:mixing_iterative_base} to substitute for $P(s,s+u)$ on the right-hand side of \eqref{eqn:mixing_iterative_base}), we obtain
\begin{align*}
    P(s,s+w) &= e^{-\maxexitrate w} I + \int_s^{s+w} \left( e^{-\maxexitrate (u-s)} + \int_s^u P(s,v) e^{\maxexitrate (v-u)} \maxexitrate B(v) dv \right) e^{-\maxexitrate(s+w-u)}\maxexitrate B(u) du \\
    &=e^{-\maxexitrate w} I + e^{-\maxexitrate w} \int_s^{s+w} \maxexitrate B(u) du + \int_s^{s+w} \int_s^u P(s,v) e^{-\maxexitrate(s+w-v)} \maxexitrate^2 B(v) B(u) dv du.
\end{align*}
 Iterating $n$ times and sending $n \to \infty$,
 \begin{align}
     P(s,s+w) &\ge \sum_{n=0}^\infty e^{-\maxexitrate w} \maxexitrate^n \int_s^{s+w} \int_s^{u_n} \cdots \int_s^{u_2} B(u_1) B(u_2) \cdots B(u_n) du_1 \cdots du_n \nonumber \\
     &\geq \sum_{n=0}^\infty e^{-\maxexitrate w} \maxexitrate^n \int_s^{s+w} \int_s^{u_n} \cdots \int_s^{u_2} G^n du_1 \cdots du_n \nonumber \\
     &= \sum_{n=0}^\infty e^{-\maxexitrate w} \frac{(\maxexitrate w)^n}{n!} G^n. \label{eqn:mixing_sum_lower_bound}
 \end{align}
 By the Perron-Frobenius theorem (see, for example, \cite[Theorem 1.5, p.~22]{seneta2006non}), there exists $\gamma > 0$ and a strictly positive column vector $v = (v(x) : x \in S)$ such that $Gv = \gamma v$. Thus,
 \begin{equation*}
     B_*(x,y) := \frac{G(x,y) v(y)}{v(x)}
 \end{equation*}
 is stochastic and irreducible. Also,
 \begin{equation*}
     G(x,y) = \gamma \frac{ v(x) B_*(x,y)}{v(y)}
 \end{equation*}
 and
 \begin{equation*}
     G^n(x,y) = \gamma^n  \frac{v(x) (B_*)^n(x,y)}{v(y)}.
 \end{equation*}
 It follows from \eqref{eqn:mixing_sum_lower_bound} that
 \begin{align*}
     P(s,s+u,x,y) &\geq \frac{v(x)}{v(y)} e^{-\maxexitrate u} \sum_{n=0}^\infty \frac{(\maxexitrate \gamma u)^n}{n!} (B_*)^n (x,y)  \\
     &\geq \frac{v(x)}{v(y)} e^{-\maxexitrate (1-\gamma) u} \sum_{n=0}^\infty e^{-\maxexitrate \gamma u}  \frac{(\maxexitrate \gamma u)^n}{n!} (B_*)^n(x,y) \\
    &\geq \frac{v(x)}{v(y)} e^{-\maxexitrate (1-\gamma) u} P_*(u)(x,y),
 \end{align*}
 where $P_*(u) = \exp(Q_* u)$ is stochastic and corresponds to the time $u$ transition matrix for a time-homogeneous irreducible Markov jump process with rate matrix $Q_* = \maxexitrate \gamma (B_* - I)$. Note that $P_*(u) \to \rankonestoch$ as $u \to \infty$ where $\rankonestoch \triangleq e \pi_*$ is a strictly positive stochastic matrix with identical rows, with $e$ denoting the column vector of all ones and $\pi_* \triangleq (\pi_*(x) : x \in S)$ denoting the unique row vector for which $\pi_* Q_* = 0$. Since $B_*$ is irreducible, there exists $k < d$ such that $\sum_{n=1}^k B_*^n$ is strictly positive; hence also $\sum_{n=0}^k (\maxexitrate \gamma u B_*)^n/n!$ is strictly positive. Thus there exists $\eta > 0$ such that
 \begin{equation*}
     P_*(u_0) = e^{-\maxexitrate \gamma u_0} \sum_{n \ge 0} \frac{(\maxexitrate \gamma u_0)^n}{n!} B_*^n \ge \eta \rankonestoch.
 \end{equation*}
 Then
 \begin{equation*}
     P_*(u_0 + t) = P_*(u_0)P_*(t) \ge \eta \rankonestoch P_*(t) = \eta \rankonestoch,
 \end{equation*}
 since $\pi_*P_*(t) = \pi_*$ for $t \ge 0$. Consequently, $P_*(u) \geq \eta \rankonestoch$ for $u \ge u_0$. Hence,
 \begin{equation*}
     P(s,s+u_0,x,y) \geq \eta \min_{y,z \in S} \left\{\frac{v(z)}{v(y)}\right\} e^{-\maxexitrate (1-\gamma) u_0} \pi_*(y)
 \end{equation*}
 uniformly in $x \in S$ and $s \geq 0$. That is, with 
 \begin{equation*}
     \delta \triangleq \eta \min_{y,z \in S} \left\{\frac {v(z)} {v(y)} \right\} e^{-\maxexitrate(1-\gamma)u_0} > 0,
 \end{equation*}
  we define the stochastic matrix $A(s, s+u_0)$ via the relationship
\begin{equation*}
    P(s,s+u_0) = \delta \rankonestoch + (1-\delta) A(s, s+u_0).
\end{equation*}
Analogously, we define $A(s+l u_0, s + (l+1) u_0)$ for all $l \in \Z_+$. Then, 
\begin{align*}
    P(s, s+2u_0) &= P(s,s+u_0)P(s+u_0,s+2u_0) \\
    &= \delta \rankonestoch + \delta (1-\delta) \rankonestoch A(s+u_0, s+2u_0) + (1-\delta)^2 A(s,s+u_0) A(s+u_0, s+2u_0).
\end{align*}
A simple induction then establishes that
\begin{equation*}
    P(s,s+ku_0) = (1-(1-\delta)^k)\rankonestoch_k(s) + (1-\delta)^k \prod_{l=0}^{k-1} A(s+lu_0, s+(l+1)u_0)
\end{equation*}
 where $\rankonestoch_k(s)$ is a stochastic matrix with identical rows.
\end{proof}
\section[The Martingale Representation for R(t)]{The Martingale Representation for $R(t)$}
\label{sec:mtg}
We now use the mixing structure identified in Section \ref{sec:mixing} to derive a martingale representation for $R(t)$. Specifically, we establish the existence of a function $\nu : \R_+ \times S \to \R$ such that
\begin{equation}
\label{eqn:mtg_rep}
    M(t) \triangleq R(t) - \E R(t) + \nu(t, X(t))
\end{equation}
is a martingale adapted to the filtration $(\mathcal{F}_t : t \geq 0)$, where $\mathcal{F}_t = \sigma(X(s) : 0 \le s \le t)$.

Before stating and proving our result concerning $M(t)$, we introduce the relevant notation. Let $\mu \triangleq (\mu(x): x \in S)$ be a row vector corresponding to the initial distribution of $X$, and let $\p_{t,x}(\cdot)$ (and $\E_{t,x}(\cdot)$) be the probability (and expectation) on the path-space of $X$, conditional on $X(t)=x$. Then, $\p(\cdot)$ and $\E(\cdot)$ are given by
\begin{equation} \label{eqn:P}
    \p(\cdot) \triangleq \sum_{x \in S} \mu(x) \p_{0,x}(\cdot)
\end{equation}
and
\begin{equation} \label{eqn:E}
    \E(\cdot) \triangleq \sum_{x \in S} \mu(x) \E_{0,x}(\cdot).
\end{equation}
With this notation in hand, recall the definition of $\gamma(t,x)$ in \ref{a4} and put
\begin{align*}
    \tilde r(t,x) &\triangleq r(t,x) + \gamma(t,x),  \\
    r_c(x) &\triangleq r(t,x) - \E r(t, X(t)), \\
    \tilde r_c(t,x) &\triangleq \tilde r(t,x) - \E \tilde r (t, X(t)), \\
    h(t_i, x) &\triangleq \int_0^\infty z H(t_i, x, dz),\\
    \HH_c(t_i,x) &\triangleq \HH(t_i,x) - \E h(t_i, X(t_i)),
\end{align*}
and define $\nu_1(t,x)$ and $\nu_2(t,x)$ via
\begin{align*}
    \nu_1(t,x) &\triangleq \int_0^\infty  \E_{t,x} \tilde r_c(t+u, X(t+u)) du, \\
     \nu_2(t,x) &\triangleq \sum_{i : t_i > t} \E_{t,x} \HH_c(t_i, X(t_i)).
\end{align*}
We set $\nu(t,x) \triangleq \nu_1(t,x) + \nu_2(t,x)$ for $t \ge 0$ and $x \in S$. 
\begin{remark} 
When $Q(t) \equiv Q$, $\GG \equiv \HH \equiv \KK \equiv 0$ a.s. for all $k \ge 1$, $\tilde r(t) \equiv \tilde r$, and $X(0)$ has the stationary distribution of $Q$, then $\nu_1(t,x)$ does not depend on $t$, and $\nu_1(t,x) \equiv \nu_1(x)$ is given by
\begin{equation*}
\nu_1(x) = \int_0^\infty \E[\tilde r_c (X(u)) \mid X(0) = x] du.
\end{equation*}
It is well known that $\nu_1 = (\nu_1(x) : x \in S)$ then satisfies \textit{Poisson's equation} given by
\begin{equation*}
    Q \nu_1 = -\tilde r_c,
\end{equation*}
where $\tilde r_c$ is the column vector with entries $(\tilde r_c (x) : x \in S)$; see \cite{glynn2024solution}. Hence, we can view $(\nu_1(t,x) : t \ge 0, x \in S)$ as a non-stationary analog to the solution of Poisson's equation that arises in the stationary setting.
\end{remark}

We now establish the martingale representation \eqref{eqn:mtg_rep} and basic properties of $\nu(t,x)$.
\begin{theorem}
\label{thm:martingale}
    Assume \ref{a1}-\ref{a5}. Then $\sup \{ |\nu(t,x)| : t \ge 0, x \in S\} < \infty$, and $(M(t) : t \ge 0)$ is a martingale adapted to $(\mathcal{F}_t : t \ge 0)$. 
\end{theorem}
\begin{proof}
We start by proving that the integral and sum defining $\nu_1$ and $\nu_2$ respectively are finite. We set $u_0 = 1$ in Proposition \ref{prop:mixing}, so that
\begin{equation}
\label{eqn:wlln_proof_mixing}
    P(s, s+k) = (1-(1-\delta)^k) \rankonestoch_k(s) + (1-\delta)^k \mathscr{A}_k(s)
\end{equation}
where $\delta > 0$, $\rankonestoch_k(s)$ is a positive stochastic matrix with identical rows and $\mathscr{A}_k(s)$ is a stochastic matrix. For an arbitrary function $f: \R_+ \times S \to \R^d$ such that $\|f\|_\infty \triangleq \sup \{ |f(t,x)| : t \geq 0, x \in S\} < \infty$,  we may use \eqref{eqn:wlln_proof_mixing} to bound
\begin{align}
    |(&P(s,s+k) f(s+k))(x) - (P(s,s+k)f(s+k))(y)| \nonumber \\&= (1-\delta)^k |(\mathscr{A}_k(s) f(s+k))(x) - (\mathscr{A}_k(s)f(s+k))(y)| \nonumber \\
    &\leq (1-\delta)^k \|f\|_\infty. \label{eqn:mixing_proof_bound}
\end{align}
Since $\| P(s+k, s+k+w) f(s+k+w)\|_\infty \le \|f\|_\infty$, it is evident that \eqref{eqn:mixing_proof_bound} implies that
\begin{equation*}
    |(P(s,s+k+w) f(s+k+w))(x) - (P(s,s+k+w)f(s+k+w))(y)| \leq (1-\delta)^k \|f\|_\infty,
\end{equation*}
and hence
\begin{align}
    &|P(t,t+w) f(t+w) (x) - \mu P(0,t+w) f(t+w)| \nonumber \\
    &=|(P(t,t+w)f(t+w))(x) - \sum_{z \in S} \mu(z) (P(0,t)P(t,t+w)f(t+w))(z)| \nonumber \\
    &=  \Big| \sum_{z \in S} \mu(z) \sum_{y \in S} P(0,t,z,y) \big((P(t,t+w)f(t+w))(x) - (P(t,t+w)f(t+w))(y)\big) \Big| \nonumber \\
    &\le \sum_{z \in S} \mu(z) \sum_{y \in S} P(0,t,z,y) (1-\delta)^{\lfloor w \rfloor} \|f\|_\infty = (1-\delta)^{\lfloor w \rfloor} \|f\|_\infty, \label{eqn:mixing_bound_proof}
\end{align}
where $\lfloor w \rfloor$ is the greatest integer less than or equal to $w$.
By \ref{a3} and \ref{a4}, $\tilde r_c$ satisfies $\|\tilde r_c\|_\infty < \infty$. Then, by the bound \eqref{eqn:mixing_bound_proof},
\begin{align*}
    |\nu_1(t,x)| &\leq \int_0^\infty \big|\E_{t,x} \tilde r_c(t+u, X(t+u))\big|du \\
     & \leq \|\tilde r_c\|_\infty \int_0^\infty (1-\delta)^{\lfloor w \rfloor}  dw = \| \tilde r_c \|_\infty / \delta < \infty,
\end{align*}
so $\nu_1(\cdot) = (\nu_1(\cdot, x): x \in S)$ is a well-defined $\R^d$-valued bounded function.
Similarly, $\|\nu_2\|_\infty \leq \|c\|_\infty / \ubar t \delta$, as at most $1/\ubar t$ points $t_i$ fall within the interval $[n,n+1)$ for any $n \ge 0$.

We now prove that the process $M = (M(t) : t \geq 0)$ defined by \eqref{eqn:mtg_rep} is a martingale. Set
\begin{align*}
    M_1(t) &= \int_0^t \tilde r_c (s,X(s)) ds + \nu_1(t,X(t)), \\
    M_2(t) &= \sum_{i=1}^{n(t)} \HH(t_i, X(t_i)) + \nu_2(t,X(t)), \\
    M_3(t) &= \sum_{j=1}^{J(t)} \GG(T_j,Y_{j-1},Y_j) - \int_0^t \gamma_1(s,X(s)) ds, \\
    M_4(t) &=\sum_{k=1}^{N(t)} \KK(\Lambda_k, X(\Lambda_k)) - \int_0^t \gamma_2(s,X(s))ds,
\end{align*}
so that $M(t) = \sum_{\ell=1}^4 M_\ell(t)$. We will prove that each $M_\ell$ is a martingale adapted to $(\mathcal{F}_t : t \ge 0)$. To start with, $M_1(\cdot)$ is clearly integrable and adapted. Also, observe that for $s, t \ge 0$,
\begin{equation*}
\begin{split}
\E[M_1(t+s) \mid \mathcal{F}_t] = M_1(t) + \E \Big[ &\int_0^s \tilde r_c (t+u, X(t+u)) du  \\
 &+ \nu_1(t+s, X(t+s)) - \nu_1(t,X(t)) \mid \mathcal{F}_t \Big].
\end{split}
\end{equation*}
But it is clear that
\begin{equation}
\label{eqn:nu_1_increment}
    \nu_1(t,x) = \int_0^s \E_{t,x} \tilde r_c(t+u, X(t+u)) du + \E_{t,x} \nu_1(t+s,X(t+s)),
\end{equation}
and hence $\E[M_1(t+s) \mid \mathcal{F}_t] = M_1(t)$, verifying the martingale property for $M_1(\cdot)$. A similar argument validates that $M_2(\cdot)$ is a martingale. That $M_3(\cdot)$ and $M_4(\cdot)$ are martingales is standard; see, for example, \cite{jacod1974multivariate}. This completes the proof.
\end{proof}

\section[The Law of Large Numbers for R(t)]{The Law of Large Numbers for $R(t)$}
\label{sec:lln}

Because $\nu(t,X(t))$ is uniformly bounded in $t$ (see Proposition \ref{prop:mixing}), proving the LLN for $R(t)$ is equivalent to establishing that $M(t) / \E R(t)$ converges to 0 as $t \to \infty$. Our first step in the analysis of $M(t)$ is computing the \textit{quadratic variation} $[M](t)$ of this martingale, where
\begin{equation}
\label{eqn:qv_def}
    [M](t) \triangleq \lim_{\|\mathcal P\| \to 0} \sum_i \big(M(\mathcal P_{i+1}) - M(\mathcal P_{i})\big)^2
\end{equation}
where the limit is taken over finite partitions $\mathcal P$ of $[0,t]$ such that $0 = \mathcal P_0 < \mathcal P_1 < \cdots < \mathcal P_n = t$, for some $n$, and $\|\mathcal P\| \triangleq \max\{\mathcal P_{i+1} - \mathcal P_{i} : 0 \leq i < n\}$.

For $t \ge 0$, let $\tilde r_c(t) \triangleq (\tilde r_c (t,x) : x \in S)$, $\gamma(t) \triangleq (\gamma(t,x): x \in S)$, $h_c(t) \triangleq (h_c(t,x) : x \in S)$, $\nu_1(t) \triangleq (\nu_1(t,x) : x \in S)$, and $\nu_2(t) \triangleq (\nu_2(t,x):  x \in S)$. Furthermore, for $i \in \{1,2\}$, let $\Delta \nu_i(T_j) \triangleq \nu_i(T_j, Y_{j}) - \nu_i(T_j,Y_{j-1})$ and $\Delta \nu(T_j) \triangleq \Delta \nu_1(T_j) + \Delta \nu_2(T_j)$. Finally, let  $h_c(t_i) \triangleq (h_c(t_i,x) : x \in S)$ where $h_c(t_i,x) \triangleq h(t_i,x) - \E h(t_i, X(t_i))$. In the following, we write $b(t) = \bigO(a(t))$ when there exist $c, t_0 < \infty$ such that $|b(t)| \le c a(t)$ for $t \ge t_0$.

\begin{proposition}
\label{prop:qv}
    Under \ref{a1}-\ref{a5},
    \begin{equation}
    \label{eqn:m_qv}
        [M](t) = \sum_{j=1}^{J(t)} \big(\GG(T_j, Y_{j-1},Y_j) + \Delta \nu(T_j)\big)^2 + \sum_{k=1}^{N(t)} \KK(\Lambda_k, X(\Lambda_k))^2.
    \end{equation}
\end{proposition}
\begin{proof}
    For every interval $(s,\tau]$ in which $X$ does not jump and $N$ does not experience an arrival, we have $X(u) = X(s{\scriptstyle+}) = X(s)$ and $N(u) = N(s{\scriptstyle+}) = N(s)$ for $s < u \le \tau$. Thus
    \begin{align}
    \begin{split}
        M(\tau)&-M(s) \\ &= \int_s^\tau (\tilde r_c(u) - \gamma(u))(X(s)) du + \nu_1(\tau, X(s)) - \nu_1(s, X(s)) \\ 
        & \quad+ \sum_{i=n(s)+1}^{n(\tau)} \big(\HH(t_i) -(P(s,t_i) h(t_i)) (X(s))\big) + \sum_{i=n(\tau)+1}^\infty ((P(\tau,t_i)-P(s,t_i))h_c(t_i))(X(s)) 
    \end{split} \displaybreak[1] \nonumber \\
    \begin{split}
        &= \int_s^\tau (\tilde r_c(u) - \gamma(u)) (X(s)) du + \int_\tau^\infty (P(\tau,u)\tilde r_c(u))(X(s)) du \\
        &\quad- \int_s^\infty (P(s,u) \tilde r_c(u))(X(s)) du + \sum_{i=n(s)+1}^{n(\tau)} \Big(\HH(t_i) -\int_s^{t_i} (Q(u)P(u,t_i)h(t_i))(X(s))du\Big) \\
        &\quad+ \big((P(\tau,\tau)-P(s,\tau)\big) \sum_{i=n(\tau)+1}^\infty \big(P(\tau,t_i) h_c(t_i))\big)(X(s)) 
    \end{split} \displaybreak[1] \nonumber \\
    \begin{split}
        &= \int_s^\tau (\tilde r_c(u) - \gamma (u)) (X(s)) du + (P(\tau,\tau) - P(s,\tau)) \int_0^\infty \big(P(\tau,\tau+u) \tilde r_c(\tau+u)\big)(X(s)) du \\
        &\quad- \int_s^\tau (P(s,u) \tilde r_c(u)) (X(s)) du + \sum_{i=n(s)+1}^{n(\tau)} \Big(\HH(t_i) -\int_s^{t_i} (Q(u)P(u,t_i)h(t_i))(X(s))du\Big) \\
        &\quad+ \big((P(\tau,\tau) - P(s,\tau)) \nu_2(\tau)\big)(X(s)) 
    \end{split} \displaybreak[1] \nonumber\\
    \begin{split}
        &= \int_s^\tau (\tilde r_c(u) - \gamma(u)) (X(s)) du + \big((P(\tau,\tau)-P(s,\tau))(\nu_1(\tau)+\nu_2(\tau))\big)(X(s)) \\
        &\quad- \int_s^\tau (P(s,u)\tilde r_c(u))(X(s)) du +  \sum_{i=n(s)+1}^{n(\tau)} \Big(\HH(t_i) -\int_s^{t_i} (Q(u)P(u,t_i)h(t_i))(X(s))du\Big)
    \end{split} \displaybreak[1]\\
    \begin{split}
        &= \int_s^\tau (\tilde r_c(u) - \gamma(u)) (X(s)) du + \int_s^\tau (Q(u)P(u,\tau)(\nu_1(\tau)+\nu_2(\tau))(X(s)) du \\
        &\quad- \int_s^\tau(P(s,u)\tilde r_c(u))(X(s)) du +  \sum_{i=n(s)+1}^{n(\tau)} \Big(\HH(t_i) -\int_s^{t_i} (Q(u)P(u,t_i)h(t_i))(X(s))du\Big).
    \end{split} \displaybreak[1] \label{eqn:integral_expression_for_M_increment}
    \end{align}
It follows by taking absolute values in \eqref{eqn:integral_expression_for_M_increment} and applying the conservative upper bounds on rewards imposed by \ref{a4} and \ref{a5} that for any $0 \le s \le \tau$,
\begin{equation}
\label{eqn:qv_a}
    |M(\tau)-M(s)| \le (\tau - s) \big(2 \|\tilde r_c\|_\infty + \|\gamma\|_\infty + \maxexitrate ( \|\nu_1\|_\infty + \|\nu_2\|_\infty) + (1/\ubar t) (1+\maxexitrate)\|c\|_\infty \big).
\end{equation}

For any partition $\mathcal P$, we may construct a refinement $\mathcal P^*$ by adding the jump points $T_1, T_2, \dots, T_{J(t)}$ and $\Lambda_1, \Lambda_2, \dots, \Lambda_{N(t)}$. Note that in any subinterval of $\mathcal P^*$ at which a $T_j$ (respectively, $\Lambda_j$) appears as a right endpoint, the variation over that interval can be bounded by $| \GG(T_j,Y_{j-1},Y_j) + \Delta \nu(T_j)|$ (respectively, $|\KK(\Lambda_j, X(\Lambda_j))|)$, plus a contribution bounded by \eqref{eqn:qv_a} with $(\tau - s) \le \|\mathcal P^*\| \le \| \mathcal P\|$. Given the bound from \eqref{eqn:qv_a} over the subintervals of the partition without jumps, we obtain the upper bound
\begin{equation*}
    \sum_{j=1}^{J(t)} (\GG(T_j, Y_{j-1},Y_j) + \Delta \nu(T_j) + \bigO(\|\mathcal P\|))^2 + \sum_{k=1}^{N(t)} (\KK(\Lambda_k, X(\Lambda_k)) + \bigO(\|\mathcal P\|))^2 + \bigO\Big(\sum_i (\mathcal P_{i+1}-\mathcal P_i)^2\Big)
\end{equation*}
on \eqref{eqn:qv_def}, and an obvious lower bound of
\begin{equation*}
    \sum_{j=1}^{J(t)} (\GG(T_j, Y_{j-1},Y_j) + \Delta \nu(T_j))^2 + \sum_{k=1}^{N(t)} \KK(\Lambda_k, X(\Lambda_k))^2.
\end{equation*}
Sending $\|\mathcal P\| \to 0$, we obtain the desired result.
\end{proof}

Under \ref{a1}-\ref{a5}, $\E R(t) = \Theta(t)$ as $t \to \infty$. Also, Proposition \ref{prop:qv} establishes that 
\begin{equation*}
    \E[M](t) = \bigO(\E J(t) + \E N(t)).
\end{equation*}
Since
\begin{align*}
    \E J(t) &= \E \int_0^t \exitrate(s,X(s)) ds, \\
    \E N(t) &= \E \int_0^t \beta(s,X(s)) ds,
\end{align*}
(see, e.g., \cite{jacod1974multivariate}), we have that $\E J(t) = \bigO(t)$ and $\E N(t) = \bigO(t)$ due to \ref{a3} and \ref{a4}. Because $\E (M(t)-M(0))^2 = \E[M](t)$ (see \cite[p.~73]{protter2005stochastic}), we have
\begin{equation*}
    \p\left( \Big| \frac{M(t)}{\E R(t)} \Big| > \epsilon \right) \le \frac{\E [M](t)}{\epsilon^2 (\E R(t) )^2} \to 0,
\end{equation*}
proving that $M(t) / \E R(t) \convp 0$ as $t \to \infty$. Hence,
\begin{equation*}
    \frac{R(t)}{\E R(t)} \convp 1
\end{equation*}
as $t \to \infty$. We now strengthen this result to a.s.~convergence.
\begin{theorem}
\label{thm:slln}
    Under \ref{a1}-\ref{a5},
    \begin{equation*}
        \frac{R(t)}{\E R(t)} \to 1 \quad \text{a.s.}
    \end{equation*}
    as $t \to \infty$.
\end{theorem}
\begin{proof}
In view of \eqref{eqn:mtg_rep} and the boundedness of $\nu(t,X(t))$, it is sufficient to prove that $M(t) / \E R(t) \to 0$ a.s.~as $t \to \infty$. Since $\E (M(n) - M(n-1))^2 = \E [M](n) - \E [M](n-1) = \bigO(1)$ by Proposition \ref{prop:qv}, we have that
\begin{equation*}
    \sum_{i=1}^\infty \frac{\E(M(i) - M(i-1))^2}{i^2} = \bigO(1) \sum_{i=1}^\infty \frac{1}{i^2} < \infty.
\end{equation*}
Hence, the Martingale Convergence Theorem \cite[p.~17]{hall1980martingale} implies that
\begin{equation*}
    \sum_{i=1}^n \frac{M(i) - M(i-1)}{i}
\end{equation*}
converges a.s.~as $n \to \infty$. Kronecker's lemma \cite[p.~31]{hall1980martingale} then implies that
\begin{equation}
\label{eqn:qv_b}
    \frac{1}{n} M(n) \to 0 \quad \text{a.s.}
\end{equation}
as $n \to \infty$. Also, the Burkholder-Davis-Gundy (BDG) inequality \cite[p.~266]{protter2005stochastic} and Proposition \ref{prop:qv} yield the inequality
\begin{align*}
    P\big( \sup_{0 \le t \le 1} \big|M(n+t)-M(n)\big| > n\big) &\le \E \Big(\sup_{0 \le t \le 1} \big|M(n+t)-M(n)\big|^2 \Big) / n^2 \\
    &= \bigO( \E [M](n+1) - \E[M](n) ) / n^2 \\
    &= \bigO(1/n^2),
\end{align*}
so the Borel-Cantelli lemma shows that
\begin{equation}
\label{eqn:qv_c}
\frac{1}{n} \sup_{0 \le t \le 1} \big| M(n+t) - M(n)\big| \to 0 \quad \text{a.s.}
\end{equation}
as $n \to \infty$. Of course, \eqref{eqn:qv_b} and \eqref{eqn:qv_c} prove that $M(t) / t \to 0$ a.s.~as $t \to \infty$, so that $M(t) / \E R(t) \to 0$ a.s.~as $t\to \infty$, proving our strong law.
\end{proof}

We will discuss the numerical computation of $\E R(t)$ in Section \ref{sec:computation_mean_var}, enabling the approximation $R(t) \overset{D}{\approx} \E R(t)$ for $t$ large.

\section[The Central Limit Theorem for R(t)]{The Central Limit Theorem for $R(t)$}
\label{sec:clt}

The natural CLT-based normal approximation for $R(t)$ is
\begin{equation}
\label{eqn:clt_51}
R(t) \overset{D}{\approx} \E R(t) + \sqrt{\Var R(t)}\,\mathcal{N}(0,1),
\end{equation}
when $t$ is large. The rigorous justification for \eqref{eqn:clt_51} draws upon a limit theorem of the form
\begin{equation}
\label{eqn:clt_52}
    \frac{R(t) - \E R(t)}{\sqrt{\Var R(t)}} \Rightarrow \mathcal{N}(0,1)
\end{equation}
as $t \to \infty$. A starting point for \eqref{eqn:clt_52} is to establish that $\Var R(t)$ grows at least at linear rate as $t \to \infty$, that is,
\begin{equation}
\label{eqn:qv_53}
    \liminf_{t \to \infty} \frac{\Var R(t)}{t} > 0.
\end{equation}
To prove \eqref{eqn:qv_53}, we assume:
\begin{assumption}
\label{a6}
    At least one of the following conditions holds:
    \begin{enumerate}[label=\roman*)]
        \item \label{cond:G_not_pathological}
        \begin{equation*}
            \ubar \eta \triangleq \inf_{t \ge 0} \sum_{(x,y) \in \mathcal{B}}  \int_0^\infty \big(z + \nu(t,y) - \nu(t,x)\big)^2 G(t,x,y,dz) > 0;
        \end{equation*}
        \item \label{cond:K_nonzero} 
        for each $x \in S$,
        \begin{equation*}
            \ubar \kappa(x) \triangleq \inf_{t \ge 0} \int_0^\infty z K(t,x,dz) > 0,
        \end{equation*} 
        and $\ubar \beta(x) > 0$.
    \end{enumerate}
\end{assumption}
\begin{remark}
    Observe that if $\inf_{t \ge 0} \sum_{(x,y) \in \mathcal B}  \Var \GG(t,x,y) > 0$, then condition \ref{cond:G_not_pathological} of \ref{a6} is valid.
\end{remark}
\begin{remark}
    Assumption \ref{a6} is an assertion that $R(t)$ has positive asymptotic variance, and is expected to be satisfied in all but pathological situations.
\end{remark}
\begin{proposition}
\label{prop:asymptotic_variance}
    Under \ref{a1}-\ref{a5},
    \begin{equation}
    \label{eqn:qv_54}
        \Var R(t) \sim \E [M](t)
    \end{equation}
    as $t \to \infty$. Furthermore, if \ref{a6} also holds,
    \begin{equation}
    \label{eqn:qv_55}
        \liminf_{t \to \infty} \frac{\E [M](t)}{t} > 0,
    \end{equation}
\end{proposition}
\begin{proof}
    We first prove \eqref{eqn:qv_55}. Note that by Proposition \ref{prop:qv},
    \begin{align*}
        \E [M](t)  &= \E \sum_{j=1}^{J(t)} (\GG(T_j, Y_{j-1}, Y_j) + \Delta \nu(T_j))^2 + \E \sum_{k=1}^{N(t)} \KK(\Lambda_k, X(\Lambda_k))^2 \displaybreak[1]\\
        &= \sum_{(x, y) \in \mathcal B}\E \sum_{j=1}^{J(t)} (\GG(T_j, x, y) + \nu(T_j, y) - \nu(T_j, x))^2 I((Y_{j-1}, Y_j) = (x,y)) \displaybreak[1]\\
        &\quad+\sum_{x \in S} E\sum_{k=1}^{N(t)} \KK(\Lambda_k, x)^2 I(X(\Lambda_k)=x) \displaybreak[1]\\
        &\ge \ubar \eta \sum_{(x,y) \in \mathcal B} \E \sum_{j=1}^{J(t)} I((Y_{j-1},Y_j) = (x,y)) + \sum_{x \in S} \ubar \kappa^2(x) \E \sum_{k=1}^{N(t)} I(X(\Lambda_k)=x) \displaybreak[1]\\
        &= \ubar \eta \sum_{(x,y) \in \mathcal B}  \int_0^t P(X(u)=x) Q(u,x,y)du + \sum_{x \in S} \ubar \kappa^2(x) \int_0^t P(X(u)=x) \beta(u,x)du
    \end{align*}
    where the last equality follows as a consequence of standard theory, e.g.~\cite{jacod1974multivariate}. Let $u_0 > 0$. Then by virtue of Proposition \ref{prop:mixing}, there exists $\delta > 0$ such that for $t \ge u_0$,
    \begin{align*}
        \E [M](t) &\ge \ubar \eta \int_{u_0}^t  \sum_{(x,y) \in \mathcal B} \delta \pi_*(x) \inf_{w \ge 0} Q(w, x, y) du +  \sum_{x \in S} \ubar \kappa^2(x) \int_{u_0}^t \delta \pi_*(x) \inf_{u \ge 0} \beta(u,x) dw \\
        &= \delta (t-{u_0}) \Big( \ubar \eta \sum_{(x,y) \in \mathcal B} \pi_*(x) \ubar q(x,y) + \sum_{x \in S} \ubar \kappa^2(x) \ubar \beta(x) \pi_*(x) \Big),
    \end{align*}
    with $\pi_*$ as in the proof of Proposition \ref{prop:mixing}.
    It follows from \ref{a3} and \ref{a6} that \eqref{eqn:qv_55} holds. As for \eqref{eqn:qv_54}, note that \eqref{eqn:mtg_rep} implies 
    \begin{equation}
    \label{eqn:qv_56}
        \E^{1/2} (M(t) - M(0))^2 \le \sqrt{\Var R(t)} + \|\nu\|_\infty
    \end{equation}
    and
    \begin{equation}
        \label{eqn:qv_57}
        \sqrt{\Var R(t)} \le \E^{1/2}(M(t) - M(0))^2 + \| \nu\|_\infty.
    \end{equation}
    Since $\E [M](t) = \E (M(t) - M(0))^2$, \eqref{eqn:qv_55}, \eqref{eqn:qv_56}, and \eqref{eqn:qv_57} yield the conclusion \eqref{eqn:qv_54}.
\end{proof}

It is now convenient to verify the LLN for $([M](t) : t \ge 0)$ that is suggested by Proposition \ref{prop:asymptotic_variance}.
\begin{proposition}
    \label{prop:slln_qv}
    Under \ref{a1}-\ref{a6},
    \begin{equation}
        \label{eqn:clt_58}
        \frac{[M](t)}{\Var R(t)} \to 1 \quad \as
    \end{equation}
    as $t \to \infty$.
\end{proposition}
\begin{proof}
    The key observation is that 
    \begin{equation*}
        [M](t) = \sum_{j=1}^{J(t)} \tilde \GG(T_j, Y_{j-1}, Y_j) + \sum_{k=1}^{N(t)} \tilde\KK(\Lambda_k, X(\Lambda_k)),
    \end{equation*}
    where $\tilde \GG(T_j,Y_{j-1},Y_j) \triangleq (\GG(T_j,Y_{j-1},Y_j) + \Delta \nu(T_j))^2$ and $\tilde \KK (\Lambda_k, X(\Lambda_k)) \triangleq \KK(\Lambda_k, X(\Lambda_k))^2$, so that $[M](t)$ is a special case of \eqref{eqn:reward_functional} with $r \equiv 0$ and $\HH \equiv 0$ (meaning that each $\HH(t_i, X(t_i))$ is deterministically zero). Hence, Theorem \ref{thm:slln} applies, thereby validating \eqref{eqn:clt_58} in view of \eqref{eqn:qv_54}.
\end{proof}

We turn next to the CLT itself.
\begin{theorem}
\label{thm:clt}
    Under \ref{a1}-\ref{a6},
    \begin{equation}
    \label{eqn:clt}
        \frac{R(t) - \E R(t)}{\sqrt{\Var R(t)}} \Rightarrow \mathcal{N}(0,1)
    \end{equation}
    as $t \to \infty$.
\end{theorem}

In order to prove Theorem \ref{thm:clt}, will use the following martingale CLT from \cite{hall1980martingale}.
\begin{proposition}[Theorem 3.4 of \cite{hall1980martingale}]
    \label{prop:hh}
    Let $(D_{ni} : n \ge 1,1 \le j \le n)$ be a family of square-integrable random variables, and suppose that $(\mathcal F_i : i \ge 1)$ is a filtration such that each $D_{ni}$ is measurable with respect to $\mathcal F_i$, and $\E[D_{ni} \mid \mathcal F_{i-1}] = 0$ a.s. Suppose that 
    \begin{enumerate}[label=\roman*)]
        \item $\max_{1 \le i \le n} |D_{ni}| \convp 0$ as $n \to \infty$, and $\E \max_{1 \le i \le n} D_{ni}^2$ is bounded in $n$; \label{cond:i}
        \item $\sum_{i=1}^n D_{ni}^2 \Rightarrow 1$ as $n \to \infty$. \label{cond:v}
    \end{enumerate}
    Then, it holds that
    $\sum_{i=1}^n D_{ni} \Rightarrow \mathcal N(0,1)$
    as $n \to \infty$.
\end{proposition}
\begin{remark}
    Theorem 3.4 of \cite{hall1980martingale} is stated in slightly greater generality than our Proposition 5 in order to account for the case when the $D_{ni}$'s are ``not quite'' martingale differences with respect to $(\mathcal F_i : i \ge 1)$. Because in our setting we will apply the theorem to $D_{ni}$'s that are genuine martingale differences, we have simplified the statement of Theorem 3.4 accordingly.
\end{remark}

\begin{proof}[Proof of Theorem \ref{thm:clt}]
    To prove the CLT, we first apply Proposition \ref{prop:hh} to the discrete time martingale $(M(n) : n \ge 1)$ at integer times. In particular, we write
    \begin{equation*}
        \frac{M(n)}{\sqrt{\E[M](n)}} =  \sum_{j=1}^n D_{nj},
    \end{equation*}
    where $D_{nj} \triangleq (M(j)-M(j-1))/\sqrt{\E [M](n)}$ for $1 \le j \le n$. In the notation of \cite[Theorem 3.4]{hall1980martingale}, we set $\mathcal{F}_{nj} \triangleq \mathcal{F}_j \triangleq \sigma(X(s): 0 \le s \le j)$ for $1 \le j \le n$, and $\mathcal{G}_n \triangleq \{\emptyset, \Omega\}$ the trivial $\sigma$-algebra, so that $\mathcal{G}_{ni} = \mathcal{F}_i$ for $1 \le i \le n$. 
    We will verify each condition of Proposition \ref{prop:hh} in order. Note that
    \begin{equation}
    \label{eqn:clt_59}
        \max_{1 \le i \le n} |D_{ni}| \le \frac{\|r_c\|_\infty + \|c\|_\infty \big( 1 / \ubar t +  \max_{1\le i\le n} (\Delta J(i) +\Delta N(i) )\big)}{\sqrt{\E [M](n)}},
    \end{equation}
    where $\Delta J(i) \triangleq J(i) - J(i-1)$ and $\Delta N(i) \triangleq N(i) - N(i-1)$ for $i \ge 1$. To bound the $\Delta J(i)$'s, set
    \begin{equation*}
        \widetilde M_i(t) \triangleq (J+N)(i-1+t)-(J+N)(i-1) - \int_{i-1}^{i-1+t} (\exitrate+\beta)(s,X(s)) ds
    \end{equation*}
    for $0 \le t \le 1$ and $i \ge 1$. Observe that $(\widetilde M_i(t) : 0 \le t \le 1)$ is a martingale adapted to $(\mathcal{F}_{i-1+t} : 0 \le t \le 1)$. Then, the quadratic variation $[\widetilde M_i](1)$ satisfies 
    \begin{equation*}
        [\widetilde M_i](1) = \Delta J(i) + \Delta N(i),
    \end{equation*}
    and
    \begin{equation*}
        \E [\widetilde M_i](1) = \int_{i-1}^i \E (\exitrate + \beta)(s, X(s)) ds \le \maxexitrate + \tilde \beta,
    \end{equation*}
    where $\tilde \beta \triangleq \max_{x \in S} \bar \beta(x)$.
    Hence, Minkowski's inequality and the BDG inequality for martingales imply that 
    \begin{align}
        &\E^{1/4} [(\Delta J(i) + \Delta N(i))^4 \mid \mathcal{F}_{i-1}] \nonumber \\
        &\le \E^{1/4} \left[ \Big( \int_{i-1}^i (\exitrate+\beta)(s, X(s)) ds\Big)^4 \mid \mathcal{F}_{i-1}\right] + \E^{1/4} [\widetilde M_i(1)^4 \mid \mathcal{F}_{i-1}] \nonumber \\
        & \le \maxexitrate + \tilde \beta + c_2^{1/4} \E^{1/4}[[\widetilde M_i](1)^2 \mid \mathcal{F}_{i-1}] \nonumber \\
        &= \maxexitrate + \tilde \beta + c_2^{1/4} \E^{1/4}[(\Delta J(i) + \Delta N(i))^2 \mid \mathcal{F}_{i-1}], \label{eqn:clt_113}
    \end{align}
    where $c_2 > 0$ is a constant appearing in the BDG inequality. But along the same lines,
    \begin{align}
        \E^{1/2} [(\Delta J(i) + \Delta N(i))^2 \mid \mathcal{F}_{i-1}] &\le \maxexitrate + \tilde \beta + \E^{1/2} [\widetilde M_i(1)^2 \mid \mathcal{F}_{i-1}] \nonumber \\
        &= \maxexitrate + \tilde \beta + \E^{1/2}[[\widetilde M_i](1) \mid \mathcal{F}_{i-1}] \le \maxexitrate + \tilde \beta + \big(\maxexitrate + \tilde \beta\big)^{1/2}. \label{eqn:clt_115}
    \end{align}
    So, $\sup_{i \ge 1} \E[(\Delta J(i) + \Delta N(i))^4 \mid \mathcal{F}_{i-1}]$ is bounded by a finite deterministic constant. Hence, \eqref{eqn:qv_55} of Proposition \ref{prop:asymptotic_variance} implies that 
    \begin{equation}
    \label{eqn:clt_510}
        \frac{1}{\E [M](n)} \E \max_{1 \le i \le n} (\Delta J(i)+\Delta N(i))^2 = \bigO(1/n) \int_0^\infty P\left(\max_{1 \le i \le n} (\Delta J(i) + \Delta N(i))^2 > x\right) dx.
    \end{equation}
    But Markov's inequality, together with our bound on $\E (\Delta J(i) + \Delta N(i))^4$, yields that
    \begin{align}
        \frac{1}{n} P\left(\max_{1 \le i \le n} (\Delta J(i) + \Delta N(i))^2 > x\right) &\le \frac{1}{n} \sum_{i=1}^n P((\Delta J(i) + \Delta N(i))^2 > x) \nonumber \\
        &\le \frac{1}{n} \sum_{i=1}^n \frac{\E (\Delta J(i) + \Delta N(i))^4}{x^2} \nonumber \\
        &\le \sup_{i \ge 1} \frac{\E (\Delta J(i) + \Delta N(i))^4}{x^2}. \label{eqn:clt_511}
    \end{align}
    With the integrable bound \eqref{eqn:clt_511}, the Dominated Convergence Theorem applied to \eqref{eqn:clt_510} shows that
    \begin{equation*}
        \frac{1}{\E[M](n)} \E \max_{1 \le i \le n} (\Delta J(i) + \Delta N(i))^2 \to 0
    \end{equation*}
    as $n \to \infty$. 
    Consequently, \eqref{eqn:clt_59} proves that $\max_{1 \le i \le n} |D_{ni}| \convp 0$ as $n \to \infty$ and $\E \max_{1 \le i \le n} D_{ni}^2$ is bounded in $n$, verifying condition \ref{cond:i}.

    Next, we verify condition \ref{cond:v}. Observe that 
    \begin{equation*}
        \E\left[D_{ni}^2 - \frac{[\widetilde M_i](1)}{\E[M](n)} \mid \mathcal F_{i-1}\right] = \frac{1}{\E[M](n)} \E\left[(M(i)-M(i-1))^2 - [\widetilde M_i](1) \mid \mathcal{F}_{i-1}\right]= 0.
    \end{equation*}
    Thus, with $\theta_i \triangleq (M(i)-M(i-1))^2 -[\widetilde M_i](1)$, we have that $\E [\theta_i \mid \mathcal{F}_{i-1}] = 0$. So,
    \begin{equation*}
        M^*(n) \triangleq \sum_{i=1}^n \theta_i
    \end{equation*}
    is a martingale adapted to $(\mathcal{F}_n : n \ge 0)$. Also,
    \eqref{eqn:clt_113} and \eqref{eqn:clt_115} imply that 
    \begin{equation*}
        \sup_{i \ge 1} \E \theta_i^2 < \infty,
    \end{equation*}
    from which it follows from Proposition \ref{prop:asymptotic_variance} that 
    \begin{equation*}
        \frac{M^*(n)}{\E[M](n)} \to 0 \quad \text{a.s.};
    \end{equation*}
    see our similar argument leading to \eqref{eqn:qv_b}. Hence, Proposition \ref{prop:slln_qv} implies that 
    \begin{equation*}
        \frac{\sum_{i=1}^n (M(i)-M(i-1))^2}{\E[M](n)} = \sum_{i=1}^n D_{ni}^2 \to 1 \quad \text{a.s.}
    \end{equation*}
    as $n \to \infty$, verifying \ref{cond:v}. We may now apply \cite[Theorem 3.4]{hall1980martingale}, thereby proving that 
    \begin{equation*}
        \frac{M(n)}{\sqrt{\E [M](n)}} \Rightarrow \mathcal{N}(0,1)
    \end{equation*}
    as $n \to \infty$, so that
    \begin{equation*}
        \frac{M(n)}{\sqrt{\Var R(n)}} \Rightarrow \mathcal{N}(0,1)
    \end{equation*}
    as $n \to \infty$ by \eqref{eqn:qv_54} of Proposition \ref{prop:asymptotic_variance}.

    Since $\E \big| M(t)-M(\lfloor t \rfloor) \big| = \bigO\big(\E (J+N)(t) - \E (J+N)(\lfloor t\rfloor)\big) = \bigO(1)$ while $\Var R(t) = \Theta(t)$, it follows that 
    \begin{equation*}
        \frac{M(t)}{\sqrt {\Var R(t)}} \Rightarrow \mathcal{N}(0,1).
    \end{equation*}
    Finally, the boundedness of $\nu(t,X(t))$ yields the CLT \eqref{eqn:clt}. 
\end{proof}
\begin{remark}
    In Section \ref{sec:computation_mean_var}, we discuss the computation of $\Var R(t)$, thereby permitting \eqref{eqn:clt} to be used to obtain distributional approximations for $R(t)$.
\end{remark}

Although the form of $\nu(t,x)$ makes \ref{a6} somewhat difficult to verify directly in the case that $\KK \equiv 0$ and the jump-time lump-sum rewards $\GG$ are deterministic, the following condition suffices to verify \ref{a6}.

\begin{assumption}
\label{a7}
    Assume that $\GG(t,x,y)$ is deterministic for all $t\ge 0$, $x,y \in S$, that is, there exists $g(t,x,y) \in \R$ such that $G(t,x,y,z) = I(g(t,x,y) \le z)$, and $\ubar g \triangleq \inf\{g(t,x,y) : t \ge 0, x,y \in S\} > 0$.
\end{assumption}
\begin{proposition}
\label{prop:nonzero_asymptotic_variance_alt}
Suppose \ref{a1}-\ref{a5}. Then \ref{a7} implies \ref{a6}.
\end{proposition}
\begin{proof}
Let $\Delta \nu(s,x,y) \triangleq \nu(s,y) - \nu(s,x)$. For any time $s \ge 0$, if $\Delta \nu(s,x,y) = 0$ for all $x,y \in S$, then
\begin{equation*}
    \sum_{(x,y) \in \mathcal B}(g(s,x,y) + \Delta \nu(s,x,y))^2 \ge \ubar g^2.
\end{equation*}
Otherwise, there is at least one pair $(x_0,y_0) \in \mathcal B$ for which $\Delta \nu(s,x,y) \ne 0$. Hence, by \ref{a2}, we may find a closed walk $x_0,x_1,x_2,\dots,x_n$ on the graph $\mathcal{B}$; that is, $x_n = x_0$ and $(x_j,x_{j+1}) \in \mathcal{B}$ for $j=0,\dots, n-1$. Necessarily, $\sum_{j=0}^{n-1} \Delta \nu(t,x_j,x_{j+1})=0$ by cancellation. Since $\Delta \nu(t,x_0,y_0) \ne 0$, there must exist some $j'$ for which $\Delta \nu(t,x_{j'}, x_{{j'}+1}) < 0$. Therefore
\begin{equation*}
    \sum_{(x,y) \in \mathcal B} (g(s,x,y) + \Delta \nu(s,x,y))^2 \ge (g(s,x_{j'},x_{j'+1}))^2 \ge \ubar g^2,
\end{equation*}
which verifies condition \ref{cond:G_not_pathological} of \ref{a6} since 
\begin{equation*}
    \inf_{s \ge 0} \int_0^\infty (z + \Delta \nu(s,x,y))^2 G(s,x,y,dz) = \inf_{s \ge 0} (g(s,x,y) + \Delta \nu(s,x,y))^2 \ge \ubar g^2 > 0
\end{equation*} 
under \ref{a7}.
\end{proof}

\begin{remark}
    The technical assumption \ref{a6} ensures that the asymptotic normalized variance is nonzero in \eqref{eqn:clt}. This condition arises because the time-varying variance term must appear in the denominator on the left-hand side of \eqref{eqn:clt} in our non-stationary setting. On the other hand, in the stationary setting, the time-average variance constant can be moved to the right-hand side of \eqref{eqn:clt}, where it appears in the numerator. As a consequence, in such stationary CLT's, one need not assume that the time-average variance constant is necessarily nonzero.
\end{remark}

\section[Computation of E R(t) and Var R(t)]{ Computation of $\E R(t)$ and $\Var R(t)$}
\label{sec:computation_mean_var}
In order to use the LLN and CLT approximations for $R(t)$, we need to compute $\E R(t)$ and $\Var R(t)$. Note that
\begin{align}
\label{eqn:comp_expectation_R}
    \E R(t) &= \int_0^t \E \tilde r(X(u))du + \sum_{i=1}^{n(t)} \E \HH(t_i, X(t_i)) \\
    &= \int_0^t \mu P(0,u) \tilde r(u) du + \sum_{i=1}^{n(t)} \mu P(0,t_i) h(t_i), \nonumber
\end{align}
recalling that $\mu = (\mu(x): x \in S)$ was defined as the distribution of $X(0)$.

We will now show that $\E_{0,x} R(t)$ can be computed as the solution of an integral (or, equivalently, differential) equation. Recall that $\E_{0,x}$ was defined above as the expectation with respect to probability measure on paths of $X$ conditional on $X(0) = x$.
For $0 \le s \le t$, we define the vector $m(s) = (m(s,x): x \in S)$ by
\begin{equation*}
    m(s) \triangleq \int_s^t P(s,u) \tilde r(u) du + \sum_{s \le t_i \le t} P(s,t_i) h(t_i).
\end{equation*}
Note that $\E_{0,x} R(t) = m(0,x)$, so that by \eqref{eqn:E}, we have that $\E R(t) = \mu m(0)$.
Let $n_0 \triangleq n(t) + 1$, and, for the purposes of this section, put $t_0 = 0$, $t_{n_0} = t$, and $h(0) \equiv h(t_n) \equiv 0$. Adding the points $t_0, t_{n_0}$ to the collection of $t_i$'s simplifies our computational discussion.

\begin{proposition}
\label{prop:integral_equations_m}
Assume \ref{a1}-\ref{a5}. For $1 \le i \le n_0$ and $t_{i-1} < s \le t_i$,
\begin{equation*}
    m(s) = \int_s^{t_i} \tilde r(u) du + \int_s^{t_i} Q(u) m(u) du + m(t_i),
\end{equation*}
where $m(t_{n_0}) = 0$.
Furthermore, $m(\cdot)$ is a.e.~differentiable, and for $0 \le s < t_i - t_{i-1}$,
\begin{equation}
\label{eqn:m_ode}
    \dd s m(t_i-s) = \tilde r(t_i-s) + Q(t_i-s)m(t_i-s) \quad \text{a.e.}
\end{equation}
and $m(t_i) = h(t_i)+m(t_i {\scriptstyle+})$.
\end{proposition}
\begin{remark}
    Note that \eqref{eqn:m_ode} provides a backward recursion for solving for $m(\cdot)$. In particular, starting from $t_{n_0} = t$, we use the differential equation \eqref{eqn:m_ode} to compute $m$ over $(t_{i-1}, t]$ (with $m(t{\scriptstyle+})=0$). Given $i < n_0$, \eqref{eqn:m_ode} allows us to compute $m$ over $(t_{i-1}, t_i]$ from $m(t_i{\scriptstyle+})$ and $h(t_i)$. 
\end{remark}
\begin{proof}
    Fix $1 \le i \le n_0$. Observe that the backward equations yield, for $s \in (t_{i-1}, t_i]$,
    \begin{align*}
        m(s) &= \int_s^{t_i} P(s,u) \tilde r(u) du  + P(s,t_i) m(t_i) \displaybreak[1]\\
        &= \int_s^{t_i}\Big(P(u,u) + \int_s^u Q(v) P(v,u) dv\Big) \tilde r(u) du + \Big(P(t_i,t_i)+\int_s^{t_i} Q(u)P(u,t_i)du\Big)m(t_i)  \displaybreak[1]\\
        &= \int_s^{t_i} \tilde r(u) du + m(t_i) + \int_s^{t_i} Q(w) \Big(\int_w^{t_i} P(w,u) \tilde r(u) du + P(w,t_i)m(t_i)\Big) dw \displaybreak[1]\\
        &= \int_s^{t_i} \tilde r(u)du + \int_s^{t_i} Q(w) m(w) dw + m(t_i).
    \end{align*}
    It follows from this representation that $m(\cdot)$ is a.e.~differentiable, and satisfies \eqref{eqn:m_ode}.
\end{proof}
Suppose for this paragraph that $Q(s)$ and $\tilde r(s)$ are continuous in $s$ except at points in $\mathcal{D}[0,t]$. Without loss of generality, we may merge the points in $\mathcal{D}[0,t]$ into the $t_i$'s, setting $h(u_i)=0$ for $u_i \in \mathcal{D}[0,t]$. Under these conditions, the fundamental theorem of calculus implies that $m(\cdot)$ is differentiable on $(t_{i-1},t_i]$ with a left derivative at $t_i$, and $m(\cdot)$ satisfies the differential equation
\begin{equation*}
\label{eqn:comp_c}
    \dd s m(t_i-s) = \tilde r(t_i-s) + Q(t_i-s)m(t_i-s)
\end{equation*}
for $0 \le s < t_i - t_{i-1}$, subject to the boundary condition $m(t_i) = h(t_i)+m(t_i{\scriptstyle +})$. 

\begin{remark} \label{remark:prendiville}
    The equation \eqref{eqn:m_ode} is an inhomogeneous, non-autonomous first-order linear ODE. Therefore, an explicit solution of \eqref{eqn:m_ode} is given by the variation of constants formula, 
\begin{equation}
    \label{eqn:variation_of_constants}
    m(t_i - s) = P(t_i - s, t_i) m(t_i) + \int_{t_i-s}^{t_i} P(t_i-s, u) \tilde r(u)\,du \end{equation}
for $s \in (t_{i-1}, t_i]$, (e.g.~ \cite[p.~82]{hale1980ordinary}). 
Here, $P(w,u)$ is the matrix of transition probabilities obtained by solving the Kolmogorov equations (\ref{eqn:kbe_matrix}, \ref{eqn:kfe_matrix}) over the interval $[w,u]$. This formula clarifies that $m(\cdot)$ is available in closed form when the transition probabilities $P(w,u)$ are available in closed-form.
However, closed-form expressions for the transition probabilities $P(w,u)$ are only available in special cases, even for stationary MJPs. One such case is the two-state switching model on $S = \{0, 1\}$. Let $\lambda(\cdot)$ and $\mu(\cdot)$ be non-negative continuous functions on $\R_+$, and let $X_1$ be an MJP on $S$ with transition rates given by
\begin{equation}
    \label{eqn:two_state_Q}
    Q(t) = \begin{pmatrix}
        -\lambda(t) & \lambda(t) \\ 
        \mu(t) & -\mu(t)
    \end{pmatrix}.
\end{equation}
It is then straightforward to solve \eqref{eqn:kfe_diff} to obtain 
\begin{equation}
\label{eqn:two_state_transition_probs}
\begin{split}
    p_{01}(s,t) \triangleq P(s,t,0,0) &= 1-e^{-A(s,t)} \left( 1 + \int_s^t e^{A(s,u)} \lambda(u)\,du \right), \\ 
    p_{11}(s,t) \triangleq P(s,t,1,0) &= e^{-A(s,t)} \left( 1 + \int_s^t e^{A(s,u)} \mu(u)\,du \right),
\end{split}
\end{equation}
for $t \ge 0$, where $A(s,t) \triangleq \int_s^t (\lambda + \mu)(u)\,du$. 
When $2 < d < \infty$, finite state non-stationary MJPs for which transition probabilities are available in closed form are scarce. One nontrivial model with closed form transition probabilities is the \textit{Prendiville process} (first introduced by \cite{feller1939grundlagen}, see also \cite{takashima1956note,zheng1998note}), which may be viewed as an ensemble of independent two-state switching models in which $X(t)$ is the total number of such models that occupy $\{1\}$ at time $t$. The MJP $X_2$ on $S = \{0, \dots, d-1\}$ is therefore a Prendiville process if its transition rates are such that
\begin{equation} \label{eqn:prendiville_Q}
     Q(t, x, y) = \begin{dcases}
        (d-1-x) \lambda(t) & \text{if }y = x+1 \le d-1, \\
        x\mu(t) & \text{if }y = x-1 \ge 0, \\
        -((d-1-x) \lambda(t) + x\mu(t)) & \text{if }y = x, \\
        0 & \text{otherwise},
    \end{dcases}
\end{equation}
for non-negative integrable functions $\lambda(\cdot)$, $\mu(\cdot)$. Then, we have 
\begin{equation}
    P(s,t,x,\cdot) = \p(\text{Binom}(d-1-x, p_{01}(s,t)) + \text{Binom}(x, p_{11}(s,t)) \in \cdot)
\end{equation}
with $p_{01}(s,t)$ and $p_{11}(s,t)$ as in \eqref{eqn:two_state_transition_probs}; a discrete convolution then produces $P(s,t,x,\cdot)$ in closed form. As noted above, the value $X_2(t)$ may be interpreted as the number of two-state switching particles that occupy state 1 at time $t$, within an ensemble of $d$ such particles each independently following dynamics \eqref{eqn:two_state_Q}, when $X_2(0)$ particles are initially in state 1 at time 0. We refer to \cite{zheng1998note} for a fuller discussion of the closed forms available for the Prendiville process, and \cite{giorno2022time} for a recent extension of the model for which transient transition probabilities are also available in closed form.
For stationary models, closed-form expressions for $P(s,t) = \exp((t-s) Q)$ have been developed for several birth-death processes, for example, the $M/M/1/k$ queue (see \cite{takacs1962introduction}). 
In the majority of cases encountered in stochastic modeling, however, \eqref{eqn:m_ode} must be solved numerically, as we discuss in Section \ref{sec:computational_considerations}.
\end{remark}

We turn next to computing $\Var R(t) = \E R^2(t) - (\E R(t))^2$. We begin by deriving an ODE for $\E R^2(t)$. Observe that 
\begin{equation*}
\begin{split}
    \E R^2(t) &= \E \sum_{j=1}^{J(t)} \GG(T_j, Y_{j-1}, Y_j)^2 + \E \sum_{i=1}^{n(t)} \HH(t_i, X(t_i))^2 + \E \sum_{k=1}^{N(t)} \KK(\Lambda_k, X(\Lambda_k))^2 \\&\quad~+ 2 \E \int_0^t r(s,X(s))(R(t)-R(s)) ds \\
    &\quad~+ 2 \E \sum_{j=1}^\infty \GG(T_j, Y_{j-1}, Y_j) (R(t)-R(T_j)) I(T_j \le t) \\
    &\quad~+2 \E \sum_{i=1}^{n(t)} \HH(t_i,X(t_i)) (R(t)-R(t_i)) \\
    &\quad~+2 \E \sum_{k=1}^{\infty} \KK(\Lambda_k, X(\Lambda_k)) (R(t) - R(\Lambda_k)) I(\Lambda_k \le t).
\end{split}
\end{equation*}
But 
\begin{align*}
    \E [r(s,X(s))(R(t)-R(s) \mid \mathcal{F}_s] &= r(s,X(s))m(s,X(s)), \displaybreak[1] \\
    \E[\GG(T_j,Y_{j-1},Y_j)(R(t)-R(T_j))I(T_j \le t) \mid \mathcal{F}_{T_j}] &= \GG(T_j, Y_{j-1}, Y_j) m(T_j,Y_j) I(T_j \le t), \displaybreak[1]\\
    \E[\HH(t,X(t_i))(R(t)-R(t_i)) \mid \mathcal{F}_{t_i}] &= \HH(t_i, X(t_i)) m(t_i,X(t_i)), \displaybreak[1]\\
    \E[\KK(t,X(\Lambda_k))(R(t)-R(\Lambda_k)) \mid \mathcal{F}_{\Lambda_k}] &= \HH(\Lambda_k, X(\Lambda_k)) m(\Lambda_k,X(\Lambda_k)). \displaybreak[1]
\end{align*}
Hence
\begin{equation}
\label{eqn:comp_B_R2}
    \E R^2(t) = \int_0^t \E \varphi(s,X(s)) ds + \sum_{i=1}^{n(t)} \E \tilde h(t_i, X(t_i))
\end{equation}
where for $0 \le s \le t$ and $x \in S$, we define
\begin{equation*}
\begin{split}
    \varphi(s,x)& \triangleq 2r(s,x)m(s,x) + \sum_{(x,y) \in \mathcal B} Q(s,x,y) \Big(\int_0^\infty z^2G(s,x,y,dz) + 2 m(s,y) \int_0^\infty z G(s,x,y, dz)\Big) \\
    &\quad+ \beta(s,x) \Big(\int_0^\infty z^2 K(s,x,dz) + 2 m(s,x) \int_0^\infty z K(s,x,dz)\Big),
\end{split}
\end{equation*}
and
\begin{equation*}
    \tilde h(t_i, x) \triangleq \int_0^\infty z^2 H(t_i,x,dz) + 2m(t_i,x) h(t_i,x).
\end{equation*}
For $0 \le s \le t$, we define the vector $v(s) = (v(s,x) : x \in S)$ by 
\begin{equation*}
    v(s) \triangleq \int_s^t P(s,u)\varphi(u)du + \sum_{s \le t_i \le t} P(s,t_i) \tilde h(t_i).
\end{equation*}
Note that $\E_{0,x} R^2(t) = v(0,x)$, so by \eqref{eqn:E}, $\E R^2(t) = \mu v(0)$.
In view of the fact that \eqref{eqn:comp_B_R2} is identical to \eqref{eqn:comp_expectation_R} with $\varphi$ substituted for $\tilde r$ and $\tilde h$ substituted for $h$, we arrive at the following result.
\begin{proposition}
\label{prop:integral_equations_v}
    Assume \ref{a1}-\ref{a5}. For $1 \le i \le n_0$ and $t_{i-1} < s \le t_i$,
    \begin{equation*}
        v(s) = \int_s^{t_i} \varphi(u)du + \int_s^{t_i} Q(u)v(u) du + v(t_i),
    \end{equation*}
    where 
    $v(t_{n_0}) = 0$.
    Furthermore, $v(\cdot)$ is a.e.~differentiable, and for $0 \le s < t_i - t_{i-1}$,
    \begin{equation} \label{eqn:v_ode}
        \dd s v(t_i-s) = \varphi(t_i-s) + Q(t_i-s) v(t_i-s) \quad \text{a.e.},
    \end{equation}
    and $v(t_i) = \tilde h(t_i) + v(t_i{\scriptstyle +})$ for $1 \le i \le n_0$.
\end{proposition}
\noindent When $Q(\cdot)$ is continuous on $(t_{i-1},t_i]$ with a left limit at $t_{i-1}$, $m(\cdot)$ and $v(\cdot)$ are differentiable on $(t_{i-1},t_i]$ and satisfy a joint linear system of ordinary differential equations, analogous to \eqref{eqn:comp_c}.

Computing the variance by $\Var R(t) = \E R^2(t) - (\E R(t))^2 = v(0) - m(0)^2$ is subject to ``catastrophic cancellation'' of significant digits (see \cite[p.~9]{higham2012accuracy}), as $v(0)$ and $m(0)^2$ are typically much larger than $\Var R(t)$. It is more numerically stable to derive integral equations for $\Var R(t)$ directly. For $0 \le s \le t$, we define the scalar quantity
\begin{equation} \label{eqn:V}
    V(s) \triangleq \mu v(s) - (\mu m(s))^2,
\end{equation}
so that $V(0) = \Var R(t)$ and $V(t_{n_0}) = 0$.
\begin{proposition} \label{prop:integral_equations_c}
    Assume \ref{a1}-\ref{a5}. For $1 \le i \le n_0$ and $t_{i-1} < s \le t_i$,
    \begin{equation}
        V(s) = \int_s^{t_i} \big(\mu \varphi(u) + \mu Q(u) v(u) - 2 \mu m(u) (\mu \tilde r(u) + \mu Q(u) m(u)) \big) du + V(t_i),
    \end{equation}
    where $V(t_{n_0}) = 0$. Furthermore, $V(\cdot)$ is a.e.~differentiable, and for $0 \le s < t_i - t_{i-1}$,
    \begin{equation} \label{eqn:ode_V}
        \frac{d}{ds} V(t_i - s) = \mu \varphi(t_i - s) + \mu Q(t_i - s) v(t_i - s) - 2 \mu m(t_i - s) (\mu \tilde r(t_i - s) + \mu Q(t_i - s) m(t_i - s)) \quad \text{a.e.},
    \end{equation}
    where $V(t_i) = \mu \tilde h(t_i) - (\mu h(t_i))^2  + V(t_i {\scriptstyle +})$ for $1 \le i < n_0$.
\end{proposition}
\noindent We omit the proof of Proposition \ref{prop:integral_equations_c}, as it is straightforward and follows from the method used to prove Propositions \ref{prop:integral_equations_m} and \ref{prop:integral_equations_v}.

\begin{remark}
Differential equations for $m$ and $v$ are known in the life insurance literature (see \cite{norberg1992hattendorff, norberg1995differential}) under the name of the Thiele and Hattendorff differential equations, respectively. 
In particular, \cite{norberg1995differential} derived ODEs for $m$ and $v$, that allow for non-stationary rates, deterministic rewards, jump-time lump-sum deterministic rewards, and scheduled lump-sum deterministic rewards, subject to the requirement that $Q(\cdot)$, $r(\cdot)$, $g(\cdot)$ are piecewise continuous. Similar ODEs can be found in \cite{bladt2020matrix}, again requiring piecewise continuous rates. Propositions \ref{prop:integral_equations_m} and \ref{prop:integral_equations_v} generalize these results to cover random lump-sum rewards arriving at jump, ``external'', and ``scheduled'' times, and provides linear integral equation representations that apply to the case of measurable rates.
\end{remark}

\section{Computational Considerations}
\label{sec:computational_considerations}
We now discuss the complexity associated with computing $\p(R(t) > z)$ when $t$ is large and $z$ lies within a small integer multiple of $\sqrt{\Var R(t)}$ from $\E R(t)$. Suppose that our error tolerance in computing $\p(R(t) > z)$ is $\epsilon$. Based on existing Berry-Esseen and Edgeworth expansion theorems for finite state Markov chains with stationary transition probabilities \cite{nagaev1957limit,mann1996berry,dolgopyat2023berry},
we are led to expect that
$|\p(R(t) > z) - \p\big(\mathcal N(0,1) > (z - \E R(t)) / \sqrt{\Var R(t)}\big)| = \Theta ((\Var R(t))^{-1/2})$ as $t \to \infty$. Thus, we expect that the CLT approximation to $\p(R(t) > z)$ only is sensible when $\epsilon$ is of the order of $t^{-1/2}$ or larger. To compute $\p(\mathcal N(0,1) > (z - \E R(t))/\sqrt{\Var R(t)})$ to error tolerance $\epsilon$, it is easy to see by a Taylor expansion argument that $\E R(t)$ and $\Var R(t)$ must be computed to error $\Theta (\epsilon)$.

For the non-stationary setting under consideration in this paper, the two main approaches for computing $\E R(t)$ and $\Var R(t)$ to a required accuracy $\Theta(\epsilon)$ are numerically solving differential equations and Monte Carlo simulation. We discuss these two approaches in order and compare their time complexities.

Suppose that $Q(\cdot)$ is $k$-times continuously differentiable on each interval $(u_{i-1},u_i)$, with all $k$ derivatives having finite-valued right and left derivatives at $u_{i-1}$ and $u_i$, respectively. Then, we can apply a $k$th order Runge-Kutta method to solve the ODEs for $m(\cdot)$, and $v(\cdot)$ over each such interval, so that the difference increment $h$ used in the associated time-stepping solver induces an associated numerical error of order $h^k$ over each interval; see \cite[p.~160]{hairer1987solving} for a theorem on the global error analysis for $k$th order Runge-Kutta methods. To achieve an error tolerance of $\epsilon$, we must therefore take $h = \Theta(\epsilon^{1/k})$. Thus, we use $\Theta(t\epsilon^{-1/k})$ time steps to numerically compute $m(\cdot)$ and $v(\cdot)$ over $[0,t]$. At an error tolerance $\epsilon = \Theta(t^{-1/2})$, this implies $\Theta(t^{1 + 1/2k})$ time steps. Each time-step involves a bounded number of matrix-vector multiplications, leading (in the absence of sparsity) to a per step complexity of $\Theta(d^2)$ floating point operations (``flops'') and a total complexity of $\bigO(d^2 t^{1+1/2k})$ flops for our normal approximation to $\p(R(t) > z)$. If the rate matrices $Q(\cdot)$ are sufficiently sparse that the cost of matrix-vector multiplication is $\bigO(d)$, a total complexity of $\bigO(d t^{1+1/2k})$ is achievable. The complexity of computing $V(\cdot)$ by solving \eqref{eqn:ode_V} depends on the sparsity of the initial distribution $\mu$. If $\mu$ contains $O(1)$ nonzero entries, the complexity of solving \eqref{eqn:ode_V} is $O(t^{1+1/2k})$, while if $\mu$ contains $O(d)$ nonzero entries, this complexity matches that of solving for $m(\cdot)$ and $v(\cdot)$.

This can be easily compared to Monte Carlo simulation as a means of computing $\p(R(t)>z)$ to error tolerance $\epsilon = \Theta(t^{-1/2})$. The number of independent simulations of $X$ over $[0,t]$ should be of order $\Theta(\epsilon^{-2})$, in view of the square root convergence rate of the Monte Carlo method; see \cite{asmussen2007stochastic}. For our analysis here, we additionally assume that the cost of simulating the random variable $X(T_k)$ at a jump time $T_k$ is $\bigO(1)$, as is the case in many practical settings of interest (for example, birth-death processes, batch Markov arrival processes, and bulk queues). Thus each simulation over $[0,t]$ requires $\Theta((\bar \lambda + \bar \beta) t)$ computational effort (measured in flops). The factor $(\bar \beta + \bar \lambda)$ appears because the time complexity of a simulation is proportional to the number of jumps plus the number of exogenous arrivals that must be generated per unit of ``wall clock'' time. Therefore the total complexity of the Monte Carlo method needed to compute $\p(R(t) > z)$ to error $\Theta(t^{-1/2})$ is $\Theta ((\bar \lambda + \bar \beta) t \epsilon^{-2}) = \Theta((\bar \lambda + \bar \beta) t^2)$, at $\epsilon = \Theta (t^{-1/2})$. It follows that when $\epsilon = \Theta(t^{-1/2})$ and $d$ is of moderate size, our CLT approximation is considerably faster at computing $\p(R(t) > z)$ to a reasonable error tolerance (of order $t^{-1/2}$ or larger) as compared to Monte Carlo simulation. Furthermore, Monte Carlo simulation becomes slow when $(\bar \lambda + \bar \beta)$ is large. 
As an additional advantage, our computation of $m(t)$ and $v(t)$ yields $\E[R(t) \mid X(0)=x]$ and $\E[R^2(t) \mid X(0)=x]$ for each $x \in S$, whereas to compute these conditional quantities, a crude implementation of the Monte Carlo approach would require $\bigO(t^2)$ independent simulations for each starting point $x \in S$.

Alternatively, one may be interested in directly computing the quantity $\p(R(t) > z)$, for $z \in \R$. A system of linear integral equations and corresponding linear first-order integro-PDE for computing the distribution of $R(t)$, first developed in \cite{hesselager1996probability} in the case of deterministic lump-sum reward sizes, is derived as follows. 
For $0 \le s \le t$, $x \in S$, and $z \in \R$, put
\begin{equation*}
    u(s,x,z) = \p_{t-s,x}(R(t)-R(t-s)>z).
\end{equation*}
We observe that the distribution of $R(t)-R(t-s)$ is a mixture of components, one of which is a point mass corresponding to the case that the process $X$ does not jump on the interval $(t,t+s]$. Note
\begin{equation*}
    R(t)-R(t-s) = \int_{t-s}^t r(v,x) dv \triangleq \mathcal{I}(s,x)
\end{equation*}
on $\{(J+N)(t) - (J+N)(t-s)= 0 \}$, so
\begin{align*}
    u(s,x,\mathcal I(s,x)) - u(s,x,\mathcal I(s,x){\scriptstyle -}) &= -\p_{t-s,x}((J+N)(t) - (J+N)(t-s)=0) \\
    &= -\exp\Big(-\int_{t-s}^t (\exitrate + \beta)(v,x) dv\Big).
\end{align*}
For $z \ne \mathcal I(s,x)$,
\begin{align*}
    &\p_{t-s,x}(R(t)-R(t-s) > z) \displaybreak[1]\\
    &=\E_{t-s, x} u\Big(s-h,X(t-s+h), \\ &\qquad\qquad\quad z-\int_{t-s}^{t-s+h} r(v,X(v)) dv - \sum_{j=J(t-s)+1}^{J(t-s+h)} \GG(T_j, Y_{j-1}, Y_j) - \sum_{k=N(t-s)+1}^{N(t-s+h)} \KK(\Lambda_k, X(\Lambda_k))\Big) \displaybreak[1]\\
    &= \E_{t-s, x} u\big(s-h, x, z-\int_{t-s}^{t-s+h} r(v,x) dv\big) I((J+N)(t-s+h)=(J+N)(t-s)) \displaybreak[1]\\
    &\quad+\E_{t-s, x} u\big(s-h,X(t-s+h), z-\int_{t-s}^{t-s+h} r(v,X(v)) dv - \sum_{j=J(t-s)+1}^{J(t-s+h)} g(T_j, Y_{j-1}, Y_j)\big) \displaybreak[1]\\ & \quad \qquad \qquad \cdot I\big((J+N)(t-s+h)\ge (J+N)(t-s)+ 1\big) \displaybreak[1]\\
    &= \big(u(s,x,z) - h \pp s u(s,x,z) - h r(t-s,x) \pp z u(s,x,v)\big) \big(1-(\exitrate+\beta)(t-s,x) h + o(h)\big) \displaybreak[1]\\
    &\quad+h\sum_{(x,y) \in \mathcal B} Q(t-s,x,y) \int_\R u(s,x,z - w) G(t-s,x,y,dw) \\
    &\quad+h \beta(t-s,x) \int_\R u(s,x,z-w) H(t-s,x,dw) + o(h),
\end{align*}
where we write $f(h) = o(h)$ if $f(h) / h \to 0$ as $h \to 0$.
Thus, 
\begin{equation}
\label{eqn:integro_pde}
\begin{split}
    \pp s u(s,x,z) &= \sum_y Q(t-s,x,y) \int_\R \big(u(s,x,z-w)-u(s,x,z)\big) G(t-s,x,dw) \\
    &\quad+\beta(t-s,x) \int_\R \big(u(s,x,z-w) - u(s,x,z)\big) H(t-s,x,dw)\\ &\quad- r(s,x) \pp z u(s,x,z)
\end{split}
\end{equation}
for $z \ne \mathcal I(s,x)$, where we set $g(v,x,x)=0$ for $x \in S$, $v \ge 0$. To determine $u$, this equation is solved subject to
\begin{equation*}
    u(0,x,z) = \begin{dcases}
        1 \quad \text{ if } z < 0, \\ 0 \quad \text { if } z \ge 0,
    \end{dcases}
\end{equation*}
for $x \in S$. Then, $u(t,x,z) = \p_{0,x}(R(t)>z)$. A computational scheme for approximating $u$ proceeds from a discretization of \eqref{eqn:integro_pde}; for the computational issues involved we refer the reader to the discussion in \cite[section 4, pp.~39--40]{hesselager1996probability}. 

\section{The Periodic Case}
\label{sec:periodic}
An especially important setting for non-stationary models arises when it is possible to ignore secular trends affecting the system dynamics and safely assume that $Q(\cdot)$ is periodic. For example, in a customer service setting, one may model the system as periodic over one week intervals. To address this special case, we may without loss of generality assume that the period equals one time unit:
\begin{assumption}
    \label{assumption:periodic}
    For all $t \ge 0$, $Q(t+1) = Q(t)$, $r(t+1)=r(t)$, $\beta(t+1)=\beta(t)$, $G(t+1) = G(t)$, $H(t+1) = H(t)$, and $K(t+1)=K(t)$. Furthermore, the $t_i$'s are periodic with period 1, and the $m$ points $t_1, \dots, t_m$ lying in $(0,1]$ satisfy $0 < t_1 < \cdots < t_m < 1$.
\end{assumption}
\noindent We work under \ref{assumption:periodic} for the remainder of this section. We note that it is not restrictive to select the beginning of the period (i.e.~$t=0$) such that $0 < t_1 < t_m < 1$. Under \ref{assumption:periodic}, note that
\begin{equation*}
    P(n+t,n+t+s) = P(t,t+s)
\end{equation*}
for $0 \le t \le 1$, $s \ge 0$, $n \in \Z_+$. Furthermore, since $|S| < \infty$ and $Q(\cdot)$ is irreducible, $P(0,1)$ has a unique stationary distribution $\pi(0)$. In addition, $P(t,t+1)$ has a unique stationary distribution $\pi(t)$ for $0 \le t \le 1$, and $\pi(t) = \pi(0) P(0,t)$.
In this periodic setting, the martingale representation of Section 3 simplifies significantly. 

Put
\begin{align}
\label{eqn:triple_star_periodic}
    \alpha &\triangleq \int_0^1 \pi(s) \tilde r(s) ds + \sum_{i=1}^m \pi(t_i)h(t_i), \\
    \Delta R(n) &\triangleq R(n) - R(n-1) \nonumber \\
    &= \int_{n-1}^n r(s,X(s)) ds + \sum_{j=J(n-1)+1}^{J(n)} \GG(T_j, Y_{j-1}, Y_j) \nonumber \\
    &\quad+ \sum_{n-1 < t_i \le n} \HH(t_i, X(t_i)) + \sum_{k=N(n-1)+1}^{N(n)} \KK(\Lambda_k, X(\Lambda_k)), \nonumber \\
    r^* &\triangleq \int_0^1 P(0,s) \tilde r(s) ds + \sum_{i=1}^m P(0,t_i) h(t_i).\nonumber
\end{align}
Suppose $k$ is a solution to Poisson's equation for the Markov chain $(X(n): n \ge 0)$, namely
\begin{equation}
\label{eqn:double_star_periodic_poisson}
    (P(0,1)-I)k = -(r^* - \alpha e).
\end{equation}
Then, because $P(0,1)$ is aperiodic under \ref{a1}, \ref{a2} (since Proposition 1 implies it is strictly positive), $P(0,1)^n = P(0,n) \to \Pi(0)$, where $\Pi(0)$ has identical rows given by $\pi(0)$, and $k$ can be taken as
\begin{equation*}
    k = \sum_{n=0}^\infty P(0,n) r_c^*,
\end{equation*}
with $r_c^* \triangleq r^* - \alpha e$. For $t \ge 0$, put
\begin{align*}
    \rho(t) &\triangleq \int_t^{\ceil t } P(t,u)(r(u)+\gamma(u)) du + \sum_{i=n(t)+1}^{n( \ceil t)} P(t,t_i) h(t_i) - \alpha(\ceil t - t) + \sum_{j=0}^\infty P(t, \ceil t +j) r^*_c \\
    &= \int_t^{\ceil t } P(t,u) (r(u) + \gamma(u)) du + \sum_{i=n(t)+1}^{n(\ceil t)} P(t,t_i)h(t_i) - \alpha (\ceil t - t) + P(t, \ceil t)k,
\end{align*}
and note that $\rho(t) = \rho(t+1)$ for $t \ge 0$ and $\rho(n) = k$ for $n \ge 0$.
\begin{theorem}
\label{thm:clt_periodic}
    Suppose \ref{a1}-\ref{a5} and \ref{assumption:periodic}. Then,
    \begin{equation*}
        R(t) - \alpha t + \rho(t, X(t))
    \end{equation*}
    is a martingale adapted to $(\mathcal F_t : t \ge 0)$. Furthermore,
    \begin{equation*}
        \frac{R(t) - \alpha t}{\sqrt t} \Rightarrow \sigma \mathcal N(0,1)
    \end{equation*}
    as $t \to \infty$, where 
    \begin{equation*}
        \sigma^2 = \pi(0) \int_0^1 P(0,t) \xi(t) dt,
    \end{equation*}
    and $\xi(t) = (\xi(t,x): x \in S)$ has entries given by
    \begin{equation*}
        \xi(t,x) \triangleq \sum_{(x,y) \in \mathcal B} Q(t,x,y) \int_0^\infty (z + \rho(t,y) - \rho(t,x))^2 G(t,x,y,dz) + \beta(t,x) \int_0^\infty z^2 K(t,x,dz).
    \end{equation*}
\end{theorem}
\begin{proof}
    We note that for $t, s \ge 0$,
    \begin{equation*}
        \begin{split}
        \rho(t,x) &= \E[R(t+s) - R(t) - \alpha s \mid X(t) = x] + (P(t,t+s)\rho(t+s))(x).
        \end{split}
    \end{equation*}
    As with the role of \eqref{eqn:nu_1_increment} in the proof of Theorem \ref{thm:martingale}, this implies that
    \begin{equation*}
        M_p(t) \triangleq R(t) - \alpha t + \rho(t,X(t))
    \end{equation*}
    is an $\mathcal F_t$-martingale. Furthermore,
    \begin{equation*}
        [M_p](t) = \sum_{j=1}^{J(t)} (\GG(T_j,Y_{j-1},Y_j) + \rho(T_j,Y_j)-\rho(T_j,Y_{j-1}))^2 + \sum_{k=1}^{N(t)} \tilde \KK (\Lambda_k, X(\Lambda_k)),
    \end{equation*}
    and $[M_p](t) - \int_0^t \xi(s,X(s)) ds$ is also a martingale, so that
    \begin{equation*}
        \frac{1}{t} [M_p](t) \to \sigma^2 \quad \text{a.s.}
    \end{equation*}
    as $t \to \infty$. An easy application of the martingale CLT then yields the CLT for $R(t)$.
\end{proof}
\begin{remark}
    A key difference between the martingale used here and that introduced in Section \ref{sec:mtg} is that $M_p(t)$ centers $R(t)$ by $\alpha t$, whereas $M(t)$ centers $R(t)$ by $\E R(t)$. In particular, unlike $\nu(t,x)$, we cannot define
    \begin{equation*}
        \rho(t,x) = \lim_{s \to \infty} \E_{t,x} [R(t+s)-R(t)-\alpha s]
    \end{equation*}
    because this limit does not exist in general due to the periodicity of $Q(\cdot)$. Rather, we have defined $\rho(t,x)$ so that
    \begin{equation*}
        \rho(t,x) = \E_{t,x} [R(\ceil t) - R(t) - \alpha(\ceil t - t)] + (P(t, \ceil t)k)(x).
    \end{equation*}
\end{remark}
\begin{remark}
    The function $(\rho(t): t \ge 0)$ is periodic, hence it is characterized by $(\rho(t) : 0 \le t \le 1)$. Also,
    \begin{equation}
    \label{eqn:double_star_periodic_rho}
        \rho(t,x) = \E_{t,x} [R(1 - t) - R(0) - \alpha(1-t)] + (P(t,1) \rho(0))(x)
    \end{equation}
    for $0 \le t \le 1$, $x \in S$. As with $m(\cdot)$ in Section \ref{sec:computation_mean_var}, we add $t_0,t_{m+1}$ to $0 = t_0 < t_1 < \cdots t_m < t_{m+1} = 1$. We specify a function $\bar \rho$, left continuous with right limits, which satisfies an identical integral equation on each interval $(t_{i-1}, t_i]$ for $1 \le i \le m+1$, namely
    \begin{equation}
    \label{eqn:star_periodic}
        \bar \rho(t) = \int_t^{t_i} (\tilde r (u) - \alpha) du + \int_{t}^{t_i} Q(u) \bar \rho(u) du + \bar \rho(t_i)
    \end{equation}
    for $t_{i-1} < t < t_i$, where $\bar \rho(t_{m+1}) =\bar \rho(1) = \bar \rho(0)$, and $\bar \rho(t_i) = \bar \rho(t_i {\scriptstyle +}) + h(t_i)$ for $1 \le i \le m$. In other words, $\bar \rho(\cdot)$ satisfies a linear integral equation with a periodic boundary condition, and $\rho(t) = \bar \rho(t)$ for $t \notin \{t_i : i \ge 0\}$. We work with the left continuous $\bar \rho$ in order that we may numerically integrate the Kolmogorov backward equations from right to left over the intervals $(t_i, t_{i+1}]$. In addition to solving for $\bar \rho(\cdot)$, it must be recognized that \eqref{eqn:star_periodic} also includes the constant $\alpha$, which can be viewed as an unknown, along with $\bar \rho(\cdot)$, within the equation \eqref{eqn:star_periodic}. Of course, \eqref{eqn:double_star_periodic_rho} implies that
    \begin{equation*}
        \rho(0,x) = \E_{0,x} R(1) - \alpha + (P(0,1)\rho(0))(x).
    \end{equation*}
    Multiplying both sides by $\pi(0,x)$ and summing over $x$, we conclude that \eqref{eqn:double_star_periodic_rho} is solvable with $\rho(0)=\rho(1)$ only if $\alpha$ satisfies \eqref{eqn:triple_star_periodic} (since $\pi(0) \rho(0) = \pi(0) P(0,1) \rho (0)$).
\end{remark}

According to the above remark, computing $\bar \rho(\cdot)$ on $(0,1]$ along with $\alpha$ involves solving the Kolmogorov backward (integral) equations according to a periodic boundary condition, and subject to impulses at the fixed times $t_i$. To avoid the need to use a numerical solution method capable of solving ODEs with unknown parameter subject to periodic boundary conditions and impulses at fixed times (see \cite[p.~322]{ascher1988numerical}),
we instead first solve the matrix-valued backward equations
\begin{equation*}
    P(t,1) = I + \int_t^1 Q(u) P(u,1) du
\end{equation*}
for $0 \le t \le 1$. Typically this computational step is accomplished using the differential form of these equations. With $P(0,1)$ in hand, we then compute $\pi(0)$ as a probability mass function solution of $\pi(0) = \pi(0) P(0)$. We next solve for
\begin{equation*}
    m(t,x) = \E_{t,x} [R(1) - R(t)]
\end{equation*}
for $0 \le t \le 1$ and $x \in S$ subject to $m(1) = 0$; see Proposition \ref{prop:integral_equations_m}. The quantity $\alpha$ can then be computed as 
\begin{equation*}
    \alpha = \pi(0) m(0).
\end{equation*}
Noting that $r^* = m(0)$, the next step is solving \eqref{eqn:double_star_periodic_poisson} for $k$. This allows us to then calculate $\rho(t)$ over $t \in [0,1]$ via
\begin{equation*}
    \rho(t) = m(t) - \alpha(1-t) + P(t,1) k
\end{equation*}
for $0 \le t \le 1$. (If the matrices $(P(t,1) : 0 < t < 1)$ were not saved while computing $P(0,1)$, we can compute $P(t,1) k$ by solving the backward equations.) Then, $(\xi(t): 0 \le t \le 1)$ is easily calculated, after which
\begin{equation*}
    \chi(t) \triangleq \int_t^1 P(t,u) \xi(u) du
\end{equation*}
can be computed over $t \in [0,1]$ by again solving the backward equations
\begin{equation*}
    \chi(t) = \int_t^1 \xi(u) du + \int_t^1 Q(u) \chi(u)du
\end{equation*}
for $0 \le t \le 1$. Finally, $\sigma^2 = \pi(0) \chi(0)$.
\begin{remark}
    Note that computing $\alpha$ and $\sigma^2$ via this approach avoids the need to calculate $\E R(t)$ and $\Var R(t)$ as in our earlier Theorem \ref{thm:clt}. 
\end{remark}

This yields a computational complexity for computing the centering and scaling constants in our CLT approximation to $R(t)$ that is independent of $t$. The periodic approach followed in this section requires the computation of $P(0,1)$, which entails solving a matrix-valued differential equation. By contrast, Theorem \ref{thm:clt} involves solving vector-valued differential equations for $\E R(t)$ and $\Var R(t)$. Thus, if $d$ is large relative to $t$, the CLT of Theorem \ref{thm:clt} may be preferable to the periodic CLT of Theorem \ref{thm:clt_periodic}. However, when $t$ is large, the periodic approach introduced here is typically more efficient. 

\section{Service Systems with Resetting}
\label{sec:reset}
So far in this paper we have developed results that apply to service systems in which there are no regularly scheduled times at which all the work present in the system is cleared. For example, manufacturing facilities and continuously operating call centers do not deterministically clear all their accumulated work at regular intervals. This framework is relevant for many, but not all, operation settings. 

A common alternative is a service system in which the system is ``reset'' or ``cleared'' at fixed times regardless of the state it occupied immediately prior to the resetting time. For example, the checkout at a grocery store is cleared at the end of each day, as is the security area of an airport when it closes for the day. However, the dynamics of customer arrival to the checkout throughout the day may still be subject to daily periodic fluctuations. Furthermore, secular trends in consumer behavior may imply that the queueing dynamics of the grocery checkout are not identical from one day to the next. In order to address such settings, the LLN and CLT we developed for the ``always open'' setting must be modified to allow for ``resetting'' behavior in the system.

We work under \ref{a1}-\ref{a5}, and assume that $X$ is independently reset at each time $n \in \Z_+$ according to a probability mass function $\mu_n \triangleq (\mu_n(x) : x \in S)$ on the states. In this setting, $X_n = (X(n+t) : 0 \le t < 1)$ is an independent jump process with rate matrix $(Q(n+t) : 0 \le t < 1)$ for $n \ge 0$. We can then define $R(t)$ as in Section \ref{sec:mtg} of this paper. We assume that none of the $t_i$'s are integer-valued. In this case, $R(n) = R(n {\scriptstyle +}) = R(n {\scriptstyle -})$ a.s., and we can write $R(n) = \sum_{i=1}^n \Delta R(i)$, where $\Delta R(i) = R(i) - R(i-1)$ for $i \ge 1$. Furthermore, $\E R(n) = \sum_{i=1}^n \E \Delta R(i)$ and $\Var R(n) = \sum_{i=1}^n \Var \Delta R(i)$, due to the independence created by resetting.
\begin{theorem}
    Under \ref{a1}-\ref{a5} and the additional assumption that 
    \begin{equation}
    \label{eqn:positive_var_condition}
        \inf_{t \ge 0} \Big[ \sum_{(x,y) \in \mathcal B} \Var \GG(t,x,y) + \sum_{x \in S} \Var \HH(t,x) + \sum_{x \in S} \Var \KK(t,x) \Big] > 0,
    \end{equation}
    we have
    \begin{equation}
    \label{eqn:reset_clt}
        \frac{R(n) - \E R(n)}{\sqrt {\Var R(n)}} \Rightarrow \mathcal{N}(0,1)
    \end{equation}
    as $n \to \infty$.
\end{theorem}
\begin{proof}
    Because $R(k) = \sum_{i=1}^k \Delta R(i)$, with the $\Delta R(i)$'s independent, it suffices to verify the Lyapunov condition $\lim_{k \to \infty} (\Var R(k))^{-3/2} \sum_{i=1}^k \E |\Delta R(i)|^3$; see \cite[Theorem 27.3, p.~371]{billingsley1995probability}. Using \ref{a2} and \eqref{eqn:positive_var_condition}, it is evident from the independent resetting behavior of $X$ that there exists $c > 0$ such that $\Var \Delta R(i) > c$ for all $i \ge 1$. By \ref{a3}-\ref{a5} and routine calculations, we may bound $\E |\Delta R(i)|^3< C < \infty$ for all $i \ge 1$. Then $(\Var R(k))^{-3/2} \sum_{i=1}^k \E |\Delta R(i)|^3 = \bigO( k^{-1/2})$. Therefore by Lyapunov's CLT, we conclude that \eqref{eqn:reset_clt} holds.
\end{proof}
\begin{remark}
    The computation of $\E \Delta R(n)$ and $\Var \Delta R(n)$ in this setting proceeds as in Section \ref{sec:computation_mean_var}. We observe that the independence of the $\Delta R(n)$'s makes it trivial to parallelize the computation of $\E \Delta R(n)$ and $\Var \Delta R(n)$.
\end{remark}

\section{Numerical Examples} \label{sec:numerics}
To facilitate the study of the quality of the CLT approximation for the distribution of $R(t)$, we simulate the long-time behavior of the reward functional for several non-stationary stochastic models. We consider models with secular trends, as well as fully periodic models. 
In particular, we consider the Prendiville model discussed in Remark \ref{remark:prendiville}, the $M_t / M_t / 1 / k$ queue, and the multi-server queue with staffing changes and time-varying traffic. The transition rates and reward structures for each model are chosen so that \ref{a1}-\ref{a5} are satisfied.
For each model, we simulate 10,000 iid sample paths over a fixed interval $[0,t]$ to obtain a simulated empirical distribution for $R(t)$. We also solve the system of ODEs presented in Section \ref{sec:computation_mean_var} to obtain numerical solutions to $m(0) = \E R(t)$ and $V(0) = \Var R(t)$. We then present the coverage achieved by the normal approximation of our Theorem \ref{thm:clt}, using the mean and variance calculated via numerical solution of the ODEs in Section \ref{sec:computation_mean_var}. {Finally, we compare the computational runtimes of the simulated empirical distribution approach to the ODE solving approach based on our normal approximation.}

To simplify our discussion of the practical implementation details relevant to the computation of $\E R(t)$ and $\Var R(t)$, we combine the three coupled ODEs of Section \ref{sec:computation_mean_var} into a single $(2d+1)$-dimensional vector valued ODE, namely, for $1 \le i \le n_0$ and $s \in (t_{i-1}, t_i]$,
\begin{equation} \label{eqn:f_ode}
    \dd s y(t_i-s) = f(t_i - s, y(t_i - s)),
\end{equation}
where we define $f : \R \times \R^{2d+1}$ according to (\ref{eqn:m_ode}, \ref{eqn:v_ode}, \ref{eqn:ode_V}) by
\begin{align*}
    f(t,y)_{1:d} &= \tilde r(t, y_1) + Q(t) y_{1:d}, \\
    f(t,y)_{d+1:2d} &= \varphi(t, y_1) + Q(t) y_{d+1:2d}, \\
    f(t,y)_{2d+1} &= \mu \varphi(t, y_{1:d}) + \mu Q(t) y_{d+1:2d} - 2 \mu m(t) (\mu \tilde r(t) + \mu Q(t) y_{1:d}),
\end{align*}
where, for convenience, we use the notation $x_{a:b}$ to refer to the vector in $\R^{b-a+1}$ formed by the $a$th through the $b$th entry of a vector $x \in \R^{2d+1}$. Furthermore, we have the boundary conditions
\begin{align}
    f(t_i, y)_{1:d} &= h(t_i) + y(t_i {\scriptstyle +})_{1:d}, \label{eqn:f_m1_boundary}\\
    f(t_i, y)_{d+1:2d} &= \tilde h(t_i) + y(t_i {\scriptstyle +})_{d+1:2d}, \label{eqn:f_m2_boundary} \\
    f(t_i, y)_{2d+1} &= \mu \tilde h(t_i) - (\mu h(t_i))^2 + y(t_i {\scriptstyle +})_{2d+1} \label{eqn:f_variance_boundary},
\end{align}
according to the ODEs for $m(\cdot)$, $v(\cdot)$, and $V(\cdot)$.
With this specification, we clearly have that for every $s \in [0,t]$ that $y(s)_{1:d} = m(s)$, $y(s)_{d+1:2d} = v(s)$, and $y(s)_{2d+1} = V(s)$. The system of ODEs \eqref{eqn:f_ode} does not admit a closed-form solution in general. We therefore seek a numerical approximation to the exact solution $m(0)$, $v(0)$, and $V(0)$. To this end, we consider several numerical ODE methods to solve for $m(\cdot)$, $v(\cdot)$, and $V(\cdot)$ simultaneously.

Classical global error bounds for $p$th order Runge-Kutta methods require $p$th order continuous differentiability in the function $f$; see, for example \cite[Thm.~3.2, p.~158]{hairer1987solving}.
To handle points of non-smoothness (e.g.~discontinuities in the derivatives $f^{(k)}$ for $k = 0,1,2,3,4$) in the problem data, we identify all points at which $f$ is non-smooth and add them to the collection of $t_i$'s from Section \ref{sec:computation_mean_var}. This divides the interval $[0,t]$ into sub-intervals $(t_i, t_{i+1})$, for $i = 1, \dots, n_0$, on each of which $f$ is smooth. For example, for a model with transition rates that are continuous piecewise linear functions of $t$, we would add each of the points at which the first-order derivative of the rates are discontinuous to the collection of $t_i$'s. We then solve the ODE \eqref{eqn:f_ode} on each sub-interval, and accumulate the results according to the boundary conditions defined in (\ref{eqn:f_m1_boundary}, \ref{eqn:f_m2_boundary}, \ref{eqn:f_variance_boundary}). Special care must be taken with regard to the value of the function $f$ at the endpoints of each sub-interval, however, as we will now examine.
Let $N_i \triangleq \lceil (t_{i+1} - t_i) \cdot h^{-1} \rceil$, and $h_i \triangleq (t_{i+1} - t_i) / N_i$. For each $i = 1, \dots, n_0$, let $((t_{ij}, t_{i,j+1}] : 1 \le j \le N_i)$ be the partition of $(t_i, t_{i+1}]$ into sub-intervals such that $t_{i,j+1} = t_{ij} + h_i$, $t_{i1} = t_i$, and $t_{i,N_i+1} = t_{i+1}$. 
For each sub-interval indexed by $k = N_i, N_i-1, \dots, 1$, we consider the classic fixed step size fourth-order Runge-Kutta iteration given by 
\begin{align*}
    F_1 &= \overset{\leftarrow}{f}(t_{ik} + h_i, y_{k+1}), \displaybreak[1] \\
    F_2 &= f(t_{ik} + \tfrac{1}{2} h_i, y_{k+1} - \tfrac{1}{2} h_i F_1), \displaybreak[1]  \\
    F_3 &= f(t_{ik} + \tfrac{1}{2} h_i, y_{k+1} - \tfrac{1}{2} h_i F_2), \displaybreak[1]  \\
    F_4 &= \overset{\rightarrow}{f}(t_{ik}, y_{k+1} - h_i F_3), \displaybreak[1] \\
    y_k &= y_{k+1} - \tfrac{1}{6} h_i (F_1 + 2 F_2 + 2 F_3 + F_4),
\end{align*}
where $\overset{\leftarrow}{f}$ and $\overset{\rightarrow}{f}$ correspond to the left- and right-continuous versions of $f$, respectively. This modified version of $f$ is then smooth on the sub-interval $[t_i, t_{i+1}]$.
For the first sub-problem, the initial point $y_{N_i + 1}$ is chosen to satisfy the boundary conditions of the ODEs for $m(\cdot)$, $v(\cdot)$, and $V(\cdot)$. At the right-most endpoint $t_{n_0} = t$, corresponding to the beginning of the backward integration, we set $y_{N_{n_0} + 1} = \mathbf{0} \in \mathbb R^{2d+1}$. We refer to this procedure as the \textit{discontinuity-aware}, fixed step size fourth-order Runge-Kutta method.
Figure \ref{fig:prendiville_ablation} below depicts the numerical error for the Prendiville model that results from naively using a discretization-unaware numerical ODE method which ignores the points at which $f$ is discontinuous when discretizing, relative to the performance of discretization-aware methods with the same step size. Unlike discretization-unaware methods, our discretization-aware approach typically exhibits monotonically decreasing error as the step size $h$ tends to zero.

Numerical ODE solving methods with adaptive step size are often superior to fixed step size methods in terms of the accuracy of the solutions obtained for a given computational effort. A discontinuity-aware, adaptive step size Runge-Kutta method may be implemented in much the same way as the fixed step size counterpart above: the times at which $f$ is discontinuous are added to the collection of $t_i$'s, and for each sub-interval $[t_i, t_{i+1}]$, the initial value problem is solved for a modification of $f$ that is smooth on the sub-interval. We use the Dormand-Prince Runge-Kutta 5(4) formulae of \cite{dormand1980family} to solve the initial value problem on each sub-interval in our numerical experiments below.

The numerical ODE solving methods we propose are implemented in Python at \texttt{\url{https://github.com/montefischer/markov_reward}}, along with code to simulate the rewards generated by the three models we consider. For the benefit of future research, we have designed our code to be easily extensible to new model types. The project README contains instructions for reproducing the numerical experiments whose results are reported below. All numerical experiments were carried out on a dedicated AMD EPYC 7502 processor (2.5 GHz base clock, 32 CPU cores) running CentOS Linux 7. Our Python (version 3.12) implementation depends on the standard numerical libraries NumPy (version 2.2.6) and SciPy (version 1.16.0) and uses double-precision floating-point arithmetic.

\paragraph{Prendiville model}
Consider the Prendiville model introduced in Remark \ref{remark:prendiville}. Let $S = \{0, 1, 2, \dots, 10\}$, corresponding to an ensemble of 10 two-state switching models. We specify rates and rewards to illustrate the flexibility of our theory.
We put 
\begin{equation}
    \begin{split}
        \lambda(t) &= 2 + \tfrac{1}{2} \sin(2\pi t), \\
        \mu(t) &= 3 - 2e^{-t / 4},
    \end{split}
\end{equation}
and take $Q(t)$ as given by \eqref{eqn:prendiville_Q}.
Clearly, \ref{a1}-\ref{a3} are satisfied by this specification of $(Q(t): t \ge 0)$.
Let $a(t) = 7t - \lfloor 7 t \rfloor$ be a one-periodic sawtooth function, and let $r(t,x) = x \cdot a(t) + 0.1$.
Transitions involving a jump from state $x$ to state $x+1$ generate deterministic reward $\GG(\cdot,x, x+1) \equiv 1$, while transitions involving a jump from state $x$ to state $x-1$ generate deterministic reward $\GG(\cdot,x,x-1) \equiv 5$. This deterministic jump reward structure verifies Assumption \ref{a7} for this model. 
Exogenous rewards arrive at periodic rate $\beta(t,\cdot) \equiv \frac{1}{4} (2 + \sin(2 \pi))$. The exogenous rewards are given by $\KK(\Lambda_i, x) \eqd 2 + 3 \sum_{j=1}^x Z_{kj} + 6 \sum_{j=x+1}^{10} Z_{kj}$, where $(Z_{kj} : k \ge 0, 1 \le j \le 10)$ is a family of iid random variables such that $Z_{11} \eqd \text{Beta}(2,5)$. Finally, at scheduled times $t_i = 5i$, a state-dependent reward of $\HH(t_i, x) = x$ is accrued. This reward specification clearly satisfies \ref{a4}, and our choice of scheduled times $(t_i : i \ge 0)$ satisfies the requirement \ref{a5}. 

We simulate the model over the interval $t \in [0,256]$, and report the simulation-based computation of $\p(R(t) \le z)$, as compared to the CLT-based approximation to the same probability. We implement this by choosing $z_1$, $z_2$, $z_3$, $z_4$ so that $\p(\mathcal N(\E R(t), \Var R(t)) \le z_i) = p_i$, where $p_1 = 0.01$, $p_2 = 0.1$, $p_3 = 0.9$, and $p_4 = 0.99$. If the normal approximation is accurate, the simulation-based estimates of $\p(R(t) \le z_i)$ should be close to $p_i$ for $i=1,2,3,4$. These results are reported in Table \ref{tab:coverage_two_state}, also as a function of $t$. As expected, the simulation results confirm that the approximation becomes excellent as $t$ increases. The computational cost of solving ODEs to high accuracy is highly competitive with simulation. Parallelizing simulations across 32 cores, it took 21.5 minutes to generate the 10,000 sample paths in our computational environment, while solving ODEs to obtain $\E R(768)$ and $\Var R(768)$ in the same computational environment took only 70 seconds. 

Additionally, we plot relative error convergence results in Figure \ref{fig:prendiville_error_convergence} for discontinuity-aware fixed-step size Runge-Kutta methods of orders 2 and 4, alongside the discontinuity-aware first-order Euler method, for computing the quantity $\E R(1)$. 
It is well-known that the global error for a $p$th order Runge-Kutta method with constant step size $h$ theoretically scales as $h^p$. Figure \ref{fig:prendiville_error_convergence} shows that our discontinuity-aware methods closely follow the convergence rate predicted by theory for $p = 1$, 2, and 4.

In order to show the sensitivity of numerical ODE solving to the correct treatment of discontinuities in the function $f$, we compare the relative numerical error obtained by Runge-Kutta methods with and without an explicit treatment of the discontinuities in Figure \ref{fig:prendiville_ablation}. The fixed step size method in the figure implements the fourth-order Runge Kutta method explained above for the step size $h$, while the adaptive step size method implements the formulas of \cite{dormand1980family}, subject to the condition that the maximum allowed step size is $h$. The vertical axis plots the relative numerical error on log scale. The asymptotic behavior around a relative error of $10^{-14}$ indicates the limit of floating point precision. The importance of implementing a discontinuity-aware method is apparent from the figure.

\begin{figure}[H]
    \centering
    \includegraphics[width=0.8\linewidth]{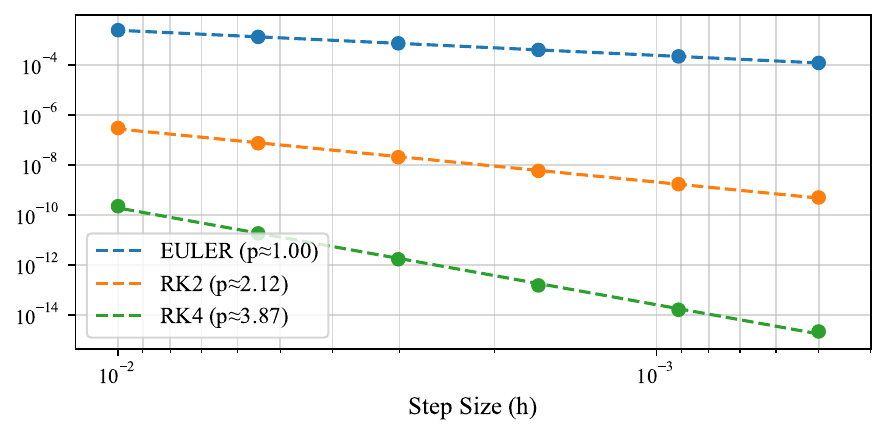}
    \caption{Error convergence plots for varying orders of numerical ODE methods for solving the ODEs for $\E R(1)$ in the Prendiville model. Dashed lines indicate linear least-squares regression lines on the log-log scale; the slopes reported in the legend, indicate the estimated order of convergence. The vertical axis plots the relative numerical error on log scale.}
    \label{fig:prendiville_error_convergence}
\end{figure}

\begin{figure}[H]
    \centering
    \includegraphics[width=0.8\linewidth]{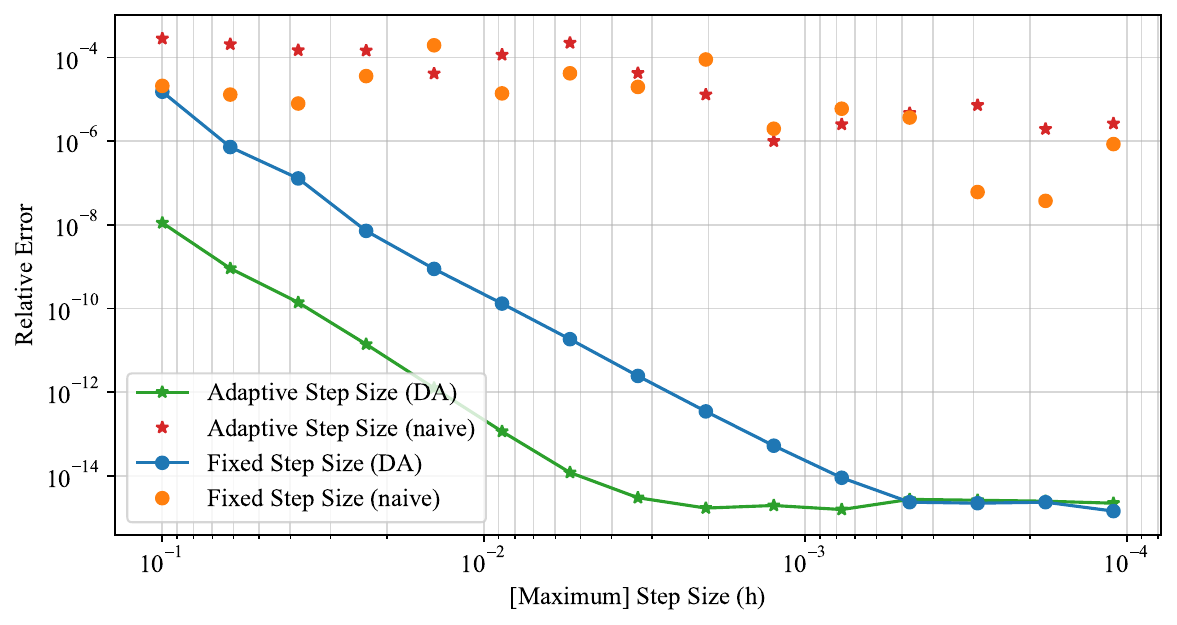}
    \caption{Comparison of Runge-Kutta methods for solving the Prendiville ODEs for $\E\,R(1)$ with and without adding the discontinuities of $\tilde r$ to the collection of $t_i$'s. The label (DA) indicates a discontinuity-aware method, whereas the label (naive) indicates a discontinuity-unaware method.}
    \label{fig:prendiville_ablation}
\end{figure}

\paragraph{Single-server queue}
We now consider the time-varying, single-server Markovian queue with finite capacity equal to 30 ($M_t / M_t / 1 / 30$ in Kendall's notation). The state space $S = \{0, 1, \dots, 30\}$ is such that the integer value of the state $x$ corresponds to the number of jobs in the system, whether in service or in the queue. 
A maximum of $30$ jobs are allowed in the system at any time. To define $Q(t)$, it suffices to specify the \textit{arrival rate} $\lambda(\cdot)$ and \textit{service rate} $\mu(\cdot)$. Let 
\begin{equation}
    \label{eqn:mm_lambda_mu}
    \begin{split}
    \lambda(t) &= 12 + 10\, \sin(\pi t), \\
    \mu(t) &= 25 + 10\, \sin(\tfrac\pi3 (t - 1/4)).
    \end{split}
\end{equation}
This specification clearly satisfies \ref{a1}-\ref{a3}. 
These dynamics describe a queueing system subject to sinusoidal arrival rate, with a sinusoidal service rate such that the model experiences periodic overloading, that is, intervals of time for which $\lambda(t) > \mu(t)$. Nonstationary queueing models with sinusoidal rates have been studied in previous work, e.g.~\cite{green1991pointwise}. 
We plot $\lambda(t)$ and $\mu(t)$ in Figure \ref{fig:mm1_queue_rates}.

\begin{figure}[H]
    \includegraphics[width=\linewidth]{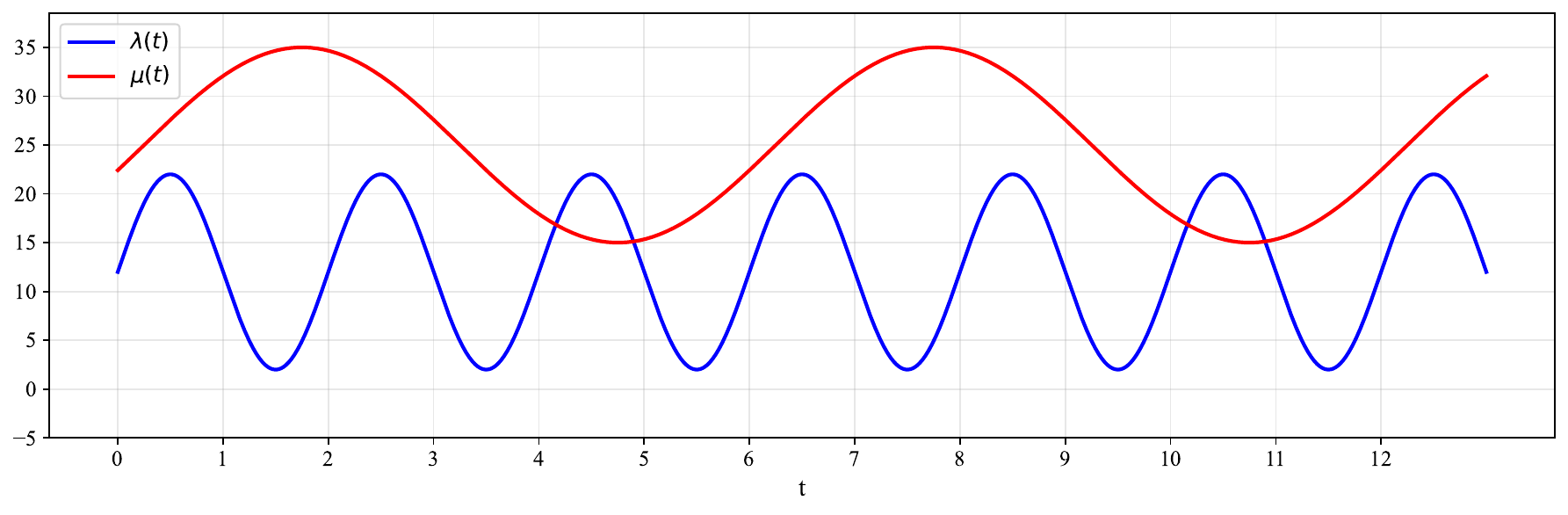}
    \caption{Rates $\lambda(t)$ and $\mu(t)$ for the $M_t / M_t / 1 / 30$ queue.}
    \label{fig:mm1_queue_rates}
\end{figure}

We specify the reward structure to correspond to a standard performance measure for queues. With $r(\cdot, x) \equiv x+1$, and all other rewards uniformly equal to zero, $R(t)/t-1$ gives the \textit{time-average number in system} over the interval $[0,t]$.

We simulate this model over the interval $t \in [0,1536]$. We simulate from an initial distribution given by a truncated geometric distribution with mean $\lambda(0) / \mu(0) \approx 0.54$, which by equation \eqref{eqn:mm_lambda_mu} approximates the stationary distribution of a time-homogeneous M/M/1 queue with arrival rate $\lambda(0)$ and service rate $\mu(0)$. In the course of running our numerical experiments, we observed that this approach, although heuristic, better illustrates our CLT result by shortening the transient period before CLT mixing is achieved relative to an initial distribution that is concentrated at the empty state. We report the comparison between the simulation-based computation of $\p(R(t) \le z)$ and the CLT-based approximation to the same probability in Table \ref{tab:coverage_mm1k}.
We note that the CLT approximation is highly accurate for large $t$. 
 Parallelizing simulations across 32 cores, generating the 10,000 sample paths took 70 minutes using our computational setup, while solving ODEs to obtain $\E R(768)$ and $\Var R(768)$ on the same processor took under 10 minutes.

\paragraph{Multi-server queue} 
In call center shift management, it can be of interest to consider multi-server queueing models that capture the degradation of agent performance over the course of a shift, with service rates recovering when the next scheduled shift arrives to refresh personnel. To this end, consider a multi-server queue with finite capacity equal to 80. For this example, we take one time unit to represent one 8 hour shift, so that the interval $[0,3]$ represents a duration of 24 hours. The transition rates $(Q(t) : t \ge 0)$ are given by a finite-capacity multi-server queue with arrival rate $\lambda(\cdot)$, and service rates $\mu(\cdot)$. We specify the arrival rate by 
\begin{equation} \label{eqn:multiserver_rates}
    % \lambda(t) = 35 + 10 \cos(2\pi t/ 3) + 10 \cos(4\pi t/3 + \pi/2) + \min\{t, 36\}.
    \lambda(t) = 35 + 10 \cos\big(\tfrac{2\pi}{3} t\big) + 10 \cos\big(\tfrac{4\pi}{3} (t + \tfrac{3}{8})\big) + \min\{t, 36\}.
\end{equation}
This choice of $\lambda$ specifies a baseline demand subject to periodic fluctuations and a piecewise linear secular trend of 3 additional arrivals per day, which ceases at the end of the twelfth day ($t=36$). A secular trend of this form may result, for example, from a corporate merger wherein traffic from a separate call center is gradually redirected to the call center being modeled, or as the result of a forecasted increase in demand from a marketing campaign. We plot the arrival rate $\lambda(t)$ in Figure \ref{fig:multiserver_rates}. We choose per-server service rates to be
\begin{equation}
    \mu(t) = 4 - \frac{x - \lfloor x \rfloor}{3},
\end{equation} 
which specifies a linear drop in server performance over the course of each shift from an initial rate of 4 customers served per shift to a final rate of $3 \tfrac23$ served per shift. Finally, we vary the number of servers on duty by shift. The first shift has 30 servers, the second has 20, and the third shift has 25. This pattern repeats for every group of three shifts. The arrival and service rates for this model are such that at some periods of time, the queue is overloaded. We do not include abandonment effects in this model. The total effective service rate is plotted alongside $\lambda(t)$ in Figure 
\ref{fig:multiserver_rates}.
It is straightforward that \ref{a1}-\ref{a3} are verified for this choice of rates.

\begin{figure}[H]
    \includegraphics[width=\linewidth]{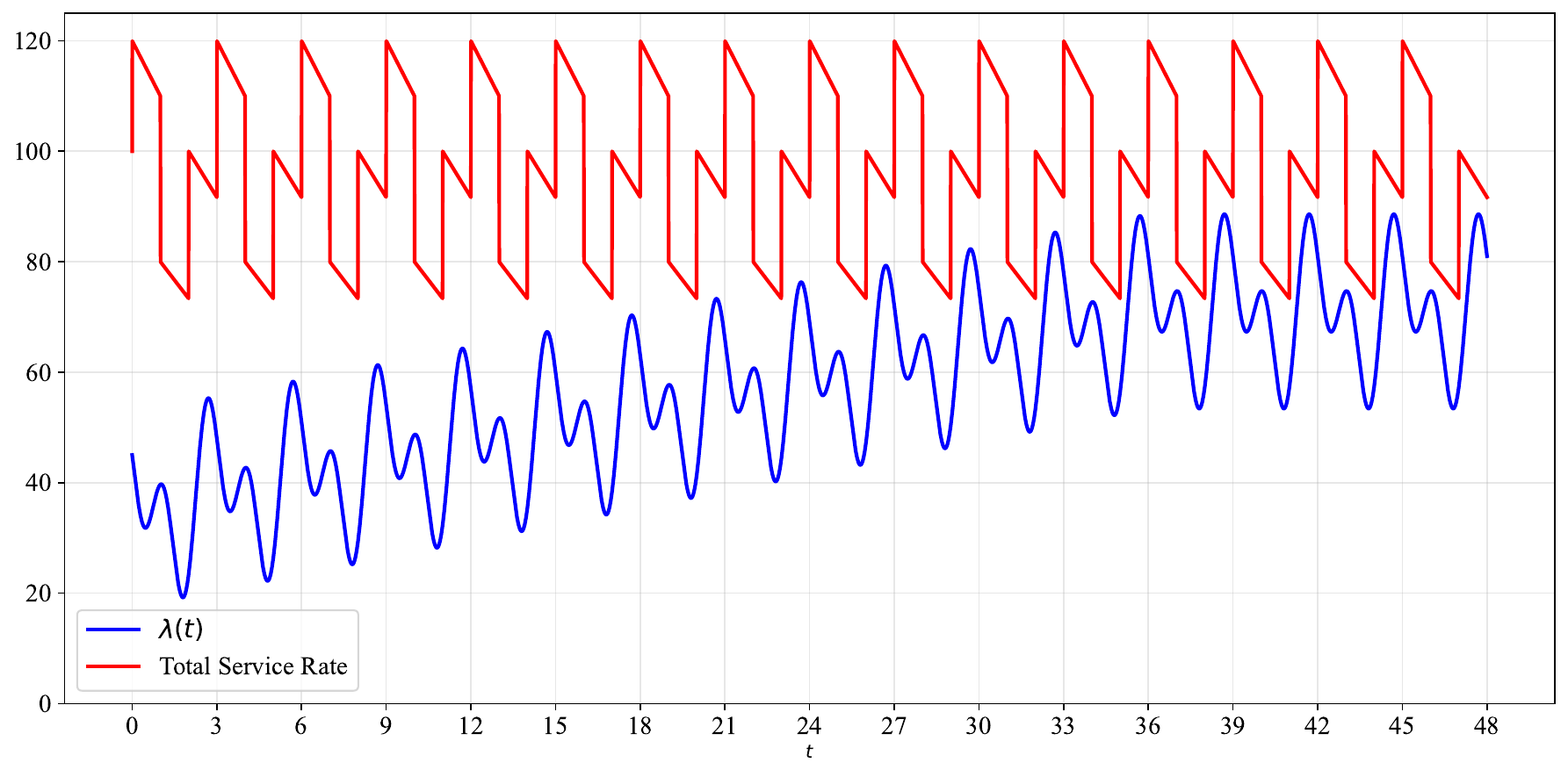}
    \caption{Arrival rate $\lambda(t)$ and total service rate for the call center multi-server queue. Each time unit corresponds to one 8-hour shift, so that e.g.~$t=3$ gives the arrival and total service rate at the end of the first 24 hours.}
    \label{fig:multiserver_rates}
\end{figure}

As with the single-server queue, we specify the reward structure to give $R(t)$ a queueing-theoretic interpretation. With $r(\cdot, x) \equiv x+1$, and all other rewards set to zero, $R(t)/t-1$ gives the time-average number in system over the interval $[0,t]$. Our assumption \ref{a4} is verified by this choice of reward; as above, we may set $t_i = i$ to satisfy \ref{a5}.

We simulate the system over the course of 256 days ($t \in [0, 768]$). Just as for the single-server queue, we simulate starting from an initial distribution corresponding to the stationary $M/M/k$ queue with arrival rate $\lambda(0)$ and service rates $\mu(0)$. We report our coverage results in Table \ref{tab:coverage_callcenter}. Parallelizing simulations across 32 cores, it took 3.6 hours to generate the 10,000 sample paths, while solving ODEs to obtain $\E R(768)$ and $\Var R(768)$ on the same processor took under 20 minutes.

These simulation studies illustrate that the approximation $R(t) \overset{D}{\approx} \mathcal N(\E R(t), \Var R(t))$ implied by Theorem \ref{thm:clt} takes effect over reasonable time scales, across a variety of non-stationary stochastic modeling scenarios. {Furthermore, comparison of the simulation cost versus the ODE solving cost in Tables \ref{tab:coverage_two_state}, \ref{tab:coverage_mm1k}, and \ref{tab:coverage_callcenter}, reveals that solving the numerical ODEs is quite computationally competitive with simulating for several models of interest.
}

\section*{Acknowledgments}
We thank the reviewers for their constructive feedback which helped to greatly enhance the presentation of this work. Some of the computing for this project was performed on the Sherlock cluster. We would like to thank Stanford University and the Stanford Research Computing Center for providing computational resources and support that contributed to these research results.

% ---------------------------------------- %

\begin{landscape}
\begin{table}[ht]
\centering
\resizebox{1.5\textwidth}{!}{%
\begin{tabular}{lcccccccc}
 & $t$ & \multicolumn{7}{c}{} \\
\cmidrule(lr){2-9}
 & 0.0625 & 0.125 & 0.25 & 1 & 4 & 16 & 64 & 256 \\
\midrule
\multicolumn{9}{l}{\textbf{Coverage probabilities } $\hat F_{t,p} = \p(R(t) \le x_{t,p})$} \\
\cmidrule(lr){1-9}
0.01 & $0.0000 \pm 0.0000$ & $0.0000 \pm 0.0000$ & $0.0000 \pm 0.0000$ & $0.0000 \pm 0.0000$ & $0.0013 \pm 0.0007$ & $0.0066 \pm 0.0016$ & $0.0075 \pm 0.0017$ & $0.0096 \pm 0.0019$ \\
0.05 & $0.0000 \pm 0.0000$ & $0.0000 \pm 0.0000$ & $0.0000 \pm 0.0000$ & $0.0013 \pm 0.0007$ & $0.0312 \pm 0.0035$ & $0.0424 \pm 0.0040$ & $0.0443 \pm 0.0041$ & $0.0530 \pm 0.0045$ \\
0.50 & $0.5914 \pm 0.0151$ & $0.6973 \pm 0.0164$ & $0.6991 \pm 0.0164$ & $0.5527 \pm 0.0146$ & $0.5283 \pm 0.0142$ & $0.5173 \pm 0.0141$ & $0.5051 \pm 0.0139$ & $0.4951 \pm 0.0138$ \\
0.95 & $0.9375 \pm 0.0190$ & $0.8822 \pm 0.0184$ & $0.9314 \pm 0.0189$ & $0.9254 \pm 0.0189$ & $0.9383 \pm 0.0190$ & $0.9457 \pm 0.0191$ & $0.9491 \pm 0.0191$ & $0.9496 \pm 0.0191$ \\
0.99 & $0.9375 \pm 0.0190$ & $0.9397 \pm 0.0190$ & $0.9729 \pm 0.0193$ & $0.9747 \pm 0.0194$ & $0.9827 \pm 0.0194$ & $0.9861 \pm 0.0195$ & $0.9868 \pm 0.0195$ & $0.9896 \pm 0.0195$ \\
\midrule
\multicolumn{9}{l}{\textbf{Computational cost} (s)} \\ 
\cmidrule(lr){1-9}
MC & $0.00118 \pm 0.00003$  & $0.00216 \pm 0.00003$  & $0.00413 \pm 0.00004$  & $0.01503 \pm 0.00006$  & $0.06071 \pm 0.00016$  & $0.25288 \pm 0.00039$  & $1.02945 \pm 0.00090$  & $4.14518 \pm 0.00227$  \\
ODE & $0.14 \pm 0.02$ & $0.19 \pm 0.01$ & $0.32 \pm 0.01$ & $0.62 \pm 0.01$ & $1.98 \pm 0.06$ & $6.60 \pm 0.42$ & $22.51 \pm 3.19$ & $70.01 \pm 3.20$ \\
\bottomrule
\end{tabular}
}
\caption{Tabulation of the quantities $\hat F_{t,p} \triangleq \p(R(t) \le x_{t,p})$ for the Prendiville model, computed by 10,000 sample Monte Carlo study, where $x_{t,p}$ is such that $\p(\mathcal N(\E R(t), \Var R(t)) \le x_{t,p}) = p$. The moments $\E R(t)$ and $\Var R(t)$ are computed by numerically solving the ODEs of Section \ref{sec:computation_mean_var} using an adaptive step size with tight relative ($10^{-11}$) and absolute ($10^{-13}$) error tolerances. Simulation cost (MC) is reported per sample. All sample quantities are reported as 95\% asymptotic confidence intervals $\hat \mu \pm 1.96 \hat \sigma / n^{1/2}$, where $\hat \mu$ is the sample mean and $\hat \sigma$ is the sample standard deviation of the reported quantity, and $n$ the number of samples. Estimates for the computational cost of ODE solving were obtained using $n=10$ solves.} 
\label{tab:coverage_two_state}
\end{table}

%%%%%%%%%%%%%%%%%%%%%%%%%%%%%%%%%%%%%%%

\begin{table}[ht]
\centering
\resizebox{1.5\textwidth}{!}{%
\begin{tabular}{lccccccc}
 & $t$ & \multicolumn{6}{c}{} \\
\cmidrule(lr){2-8}
 & 0.375 & 1.5 & 6 & 24 & 96 & 384 & 1536 \\
\midrule
\multicolumn{8}{l}{\textbf{Coverage probabilities } $\hat F_{t,p} = \p(R(t) \le x_{t,p})$} \\
\cmidrule(lr){1-8}
0.01 & $0.0000 \pm 0.0000$ & $0.0000 \pm 0.0000$ & $0.0000 \pm 0.0000$ & $0.0000 \pm 0.0000$ & $0.0021 \pm 0.0009$ & $0.0067 \pm 0.0016$ & $0.0075 \pm 0.0017$ \\
0.05 & $0.0000 \pm 0.0000$ & $0.0000 \pm 0.0000$ & $0.0000 \pm 0.0000$ & $0.0189 \pm 0.0027$ & $0.0343 \pm 0.0036$ & $0.0434 \pm 0.0041$ & $0.0478 \pm 0.0043$ \\
0.50 & $0.6393 \pm 0.0157$ & $0.6452 \pm 0.0157$ & $0.5848 \pm 0.0150$ & $0.5587 \pm 0.0147$ & $0.5345 \pm 0.0143$ & $0.5204 \pm 0.0141$ & $0.5076 \pm 0.0140$ \\
0.95 & $0.9268 \pm 0.0189$ & $0.9342 \pm 0.0189$ & $0.9267 \pm 0.0189$ & $0.9351 \pm 0.0190$ & $0.9415 \pm 0.0190$ & $0.9447 \pm 0.0191$ & $0.9481 \pm 0.0191$ \\
0.99 & $0.9617 \pm 0.0192$ & $0.9606 \pm 0.0192$ & $0.9691 \pm 0.0193$ & $0.9764 \pm 0.0194$ & $0.9820 \pm 0.0194$ & $0.9857 \pm 0.0195$ & $0.9896 \pm 0.0195$ \\

\midrule
\multicolumn{8}{l}{\textbf{Computational cost} (s)} \\ 
\cmidrule(lr){1-8}
MC & $0.00458 \pm 0.00005$  & $0.01581 \pm 0.00010$  & $0.05279 \pm 0.00018$  & $0.20970 \pm 0.00037$  & $0.83532 \pm 0.00077$  & $3.34209 \pm 0.00196$  & $13.37192 \pm 0.00551$  \\
ODE & $0.75 \pm 0.25$ & $2.00 \pm 0.33$ & $5.63 \pm 0.77$ & $17.76 \pm 0.69$ & $58.55 \pm 3.50$ & $184.61 \pm 5.23$ & $574.33 \pm 3.51$ \\
\bottomrule
\end{tabular}
}
\caption{Tabulation of the quantities $\hat F_{t,p} \triangleq \p(R(t) \le x_{t,p})$ and $x_{t,p}$ for $M_t / M_t / 1 / k$ model computed by 10,000 sample Monte Carlo study, where $x_{t,p}$ is such that $\p(\mathcal N(\E R(t), \Var R(t)) \le x_{t,p}) = p$. The moments $\E R(t)$ and $\Var R(t)$ are computed by numerically solving the ODEs of Section \ref{sec:computation_mean_var} using an adaptive step size with tight relative ($10^{-11}$) and absolute ($10^{-13}$) error tolerances. Simulation cost (MC) is reported per sample. All sample quantities are reported as 95\% asymptotic confidence intervals $\hat \mu \pm 1.96 \hat \sigma / n^{1/2}$, where $\hat \mu$ is the sample mean and $\hat \sigma$ is the sample standard deviation of the reported quantity, and $n$ the number of samples. Estimates for the computational cost of ODE solving were obtained using $n=10$ solves.}
\label{tab:coverage_mm1k}
\end{table}

%%%%%%%%%%%%%%%%%%%%%%%%%%%%%%%%%%%

\begin{table}[ht]
\centering
\resizebox{1.5\textwidth}{!}{%
\begin{tabular}{lccccccc}
 & $t$ & \multicolumn{5}{c}{} \\
\cmidrule(lr){2-7}
& 0.75 & 3 & 12 & 48 & 192 & 768 \\
\midrule
\multicolumn{7}{l}{\textbf{Coverage probabilities } $\hat F_{t,p} = \p(R(t) \le x_{t,p})$} \\
\cmidrule(lr){1-7}
0.01 & $0.0029 \pm 0.0011$ & $0.0055 \pm 0.0015$ & $0.0073 \pm 0.0017$ & $0.0059 \pm 0.0015$ & $0.0080 \pm 0.0018$ & $0.0082 \pm 0.0018$ \\
0.05 & $0.0393 \pm 0.0039$ & $0.0420 \pm 0.0040$ & $0.0523 \pm 0.0045$ & $0.0429 \pm 0.0041$ & $0.0437 \pm 0.0041$ & $0.0495 \pm 0.0044$ \\
0.50 & $0.5208 \pm 0.0141$ & $0.5110 \pm 0.0140$ & $0.5003 \pm 0.0139$ & $0.5184 \pm 0.0141$ & $0.5054 \pm 0.0139$ & $0.5078 \pm 0.0140$ \\
0.95 & $0.9423 \pm 0.0190$ & $0.9489 \pm 0.0191$ & $0.9492 \pm 0.0191$ & $0.9442 \pm 0.0190$ & $0.9502 \pm 0.0191$ & $0.9522 \pm 0.0191$ \\
0.99 & $0.9842 \pm 0.0194$ & $0.9864 \pm 0.0195$ & $0.9890 \pm 0.0195$ & $0.9834 \pm 0.0194$ & $0.9888 \pm 0.0195$ & $0.9889 \pm 0.0195$ \\
\midrule
\multicolumn{7}{l}{\textbf{Computational cost} (s)} \\ 
\cmidrule(lr){1-7}
MC & $0.02028 \pm 0.00009$  & $0.08872 \pm 0.00022$  & $0.40129 \pm 0.00051$  & $2.14621 \pm 0.00141$  & $9.90459 \pm 0.00435$  & $40.93899 \pm 0.01232$  \\
ODE & $2.36 \pm 0.08$ & $7.44 \pm 0.06$ & $26.77 \pm 0.60$ & $90.43 \pm 1.41$ & $317.40 \pm 4.11$ & $1121.67 \pm 10.80$ \\

\bottomrule
\end{tabular}
}
\caption{Tabulation of the quantities $\hat F_{t,p} \triangleq \p(R(t) \le x_{t,p})$ and $x_{t,p}$ for call center model computed by 10,000 sample Monte Carlo study, where $x_{t,p}$ is such that $\p(\mathcal N(\E R(t), \Var R(t)) \le x_{t,p}) = p$. The moments $\E R(t)$ and $\Var R(t)$ are computed by numerically solving the ODEs of Section \ref{sec:computation_mean_var} using an adaptive step size with tight relative ($10^{-11}$) and absolute ($10^{-13}$) error tolerances. Simulation cost (MC) is reported per sample. All sample quantities are reported as 95\% asymptotic confidence intervals $\hat \mu \pm 1.96 \hat \sigma / n^{1/2}$, where $\hat \mu$ is the sample mean and $\hat \sigma$ is the sample standard deviation of the reported quantity, and $n$ the number of samples. Estimates for the computational cost of ODE solving were obtained using $n=10$ solves.} 
\label{tab:coverage_callcenter}
\end{table}
\end{landscape}

\printbibliography

\end{document}